\documentclass[11pt,reqno]{amsart}
\usepackage{scalerel,stackengine}
\stackMath
\usepackage{amsmath,amssymb,amsthm,mathrsfs,amsfonts,dsfont,stmaryrd} 
\usepackage{multirow}
\usepackage{amscd}   
\usepackage{fullpage}
\usepackage[all]{xy} 
\usepackage{MnSymbol,graphicx}
\usepackage{mathrsfs,enumitem}
\usepackage[titletoc,toc,title]{appendix}
\usepackage{xcolor}  
\usepackage[normalem]{ulem}
\usepackage{diagbox}
\usepackage{caption}

\usepackage[colorlinks = true,
linkcolor = blue,
urlcolor  = blue,
citecolor = blue,
anchorcolor = blue]{hyperref} 
\usepackage{graphicx}
\usepackage{color}
\usepackage{tikz, framed}
\usetikzlibrary{arrows,shapes,matrix,decorations.pathmorphing}
\usetikzlibrary{backgrounds}
\usepackage[english]{babel}

\newtheorem{Theorem}{Theorem} 
\newtheorem{Lemma}[Theorem]{Lemma}
\newtheorem{Proposition}[Theorem]{Proposition}
\newtheorem{Corollary}[Theorem]{Corollary}

\theoremstyle{definition}

\newtheorem{Definition}[Theorem]{Definition}
\newtheorem{Remark}[Theorem]{Remark}
\newtheorem{Example}[Theorem]{Example}
\newtheorem{examplex}[Theorem]{Example}
 
\newenvironment{example}  
  {\pushQED{\qed}\examplex}
  {\popQED\endexamplex}

\newcommand{\NN}{\mathbb{N}}
\newcommand{\ZZ}{\mathbb{Z}}

\newcommand{\kD}{\mathcal{D}}

\newcommand{\A}{{\mathcal{A}}}
\newcommand{\kS}{{\mathcal{S}}}
\newcommand{\I}{{\mathcal{I}}}
\newcommand{\kC}{{\mathcal{C}}}
\newcommand{\kH}{{\mathcal{H}}}

\newcommand{\kG}{{\mathcal{G}}}
\newcommand{\kL}{{\mathcal{L}}}

\newcommand{\kF}{{\mathcal{F}}}
\newcommand{\kN}{{\mathcal{N}}}

\newcommand{\FF}{\mathbb{F}}
\newcommand{\Fq}{\mathbb{F}_q}
 \newcommand{\kO}{\mathcal{O}}

\newcommand{\Fqm}{{\mathbb{F}_{q^m}}}

\DeclareMathOperator{\rk}{rank}
\DeclareMathOperator{\cl}{cl}
\DeclareMathOperator{\codim}{codim}
\DeclareMathOperator{\diag}{diag}
\DeclareMathOperator{\opp}{opp}
\DeclareMathOperator{\upp}{upp}
\DeclareMathOperator{\low}{low}

\usepackage{soul}

\newcommand\ol[1]{{\setul{-0.9em}{}\ul{#1}}}

 \def\bcB{\color{blue}}
 \def\ecr{\color{black}}

\definecolor{bcP}{rgb}{1,0,1}

\usepackage{array}
\usepackage{makecell}

\begin{document}  

\title{Constructions of new $q-$cryptomorphisms}


\author[Eimear Byrne]{Eimear Byrne}
\address{School of Mathematics and Statistics, University College Dublin, Belfield, Ireland}
\curraddr{}
\email{ebyrne@ucd.ie}
\thanks{}

\author[Michela Ceria]{Michela Ceria}
\address{
Dept. of Mechanics, Mathematics \& Management, Politecnico di Bari, Italy
}
\curraddr{  Via Orabona 4 - 70125 Bari - Italy}
\email{michela.ceria@gmail.com}
\thanks{}

\author[Relinde Jurrius]{Relinde Jurrius}
\address{Faculty of Military Science, Netherlands Defence Academy, The Netherlands}
\curraddr{}
\email{rpmj.jurrius@mindef.nl}
\thanks{}



\maketitle

 \begin{abstract}
 In the theory of classical matroids, there are several known equivalent axiomatic systems that define a matroid, which are described as matroid {\em cryptomorphisms}.
 A $q$-matroid is a $q$-analogue of a matroid where subspaces play the role of the subsets in the classical theory. 
 In this article we establish cryptomorphisms of $q$-matroids. In doing so we highlight the difference between classical theory and its $q$-analogue. 
 We introduce a comprehensive set of $q$-matroid axiom systems and show cryptomorphisms between them and existing axiom systems of a $q$-matroid. These axioms are described as the rank, closure, basis, independence, dependence, circuit, hyperplane, flat, open space, spanning space, non-spanning space, and bi-colouring axioms.
 \smallskip 
 
\end{abstract}

\section{Introduction} 
The concept of a $q$-matroid goes back to Crapo \cite{C64}, although it has only recently been taken up again as a research topic, having been rediscovered in \cite{JP18}. As the term suggests, it arises as a {\em $q$-analogue} of matroid theory, wherein subspaces play the role of the subsets in the classical theory. The definition of a $q$-matroid with respect to a rank function can be read in \cite{JP18}: a $q$-matroid consists of a vector space $E$ together with an integer-valued, {\em bounded, monotonic increasing, semi-modular} rank function on the lattice of subspaces of $E$ (see Definition \ref{rankfunction}). 
There have been a few other recent papers on this topic, especially in relation to rank metric codes \cite{gorla2019rank,GJ20,S19}.

In the theory of classical matroids, there are several known equivalent axiomatic systems that define a matroid, which are described as matroid {\em cryptomorphisms}. A full exposition of these is given in \cite{white}. This array of cryptomorphisms offers both insight to the structure of a matroid and versatility in applications: one description of a matroid may make it more amenable to a given application than another.

In this article we seek to establish a wide portfolio of cryptomorphisms of $q$-matroids. In doing so we highlight the difference between classical theory and its $q$-analogue.
Some cryptomorphisms have already been shown in \cite{JP18}. In \cite{WINEpaper1}, it was shown that the axioms defining a collection of {\em flats} defines equivalently a $q$-matroid and conversely that a $q$-matroid with a given rank function determines a collection of flats. As an application, it was shown that a $q$-Steiner system yields a $q$-matroid (in fact a $q$-{\em perfect matroid design}) determined by a collection of flats and this was used to construct new {\em subspace designs}. 

\begin{figure}[ht]
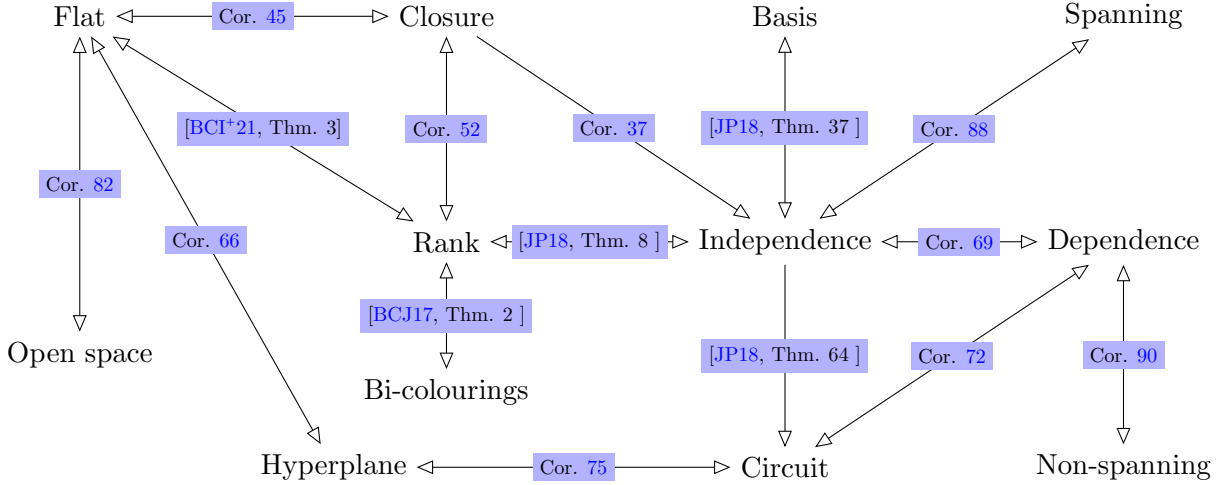

		\tikz[every node/.style={}, xscale=.75]
		{
			\node (closure) at (-2,3)  {Closure};
			\node (flats) at (-8.5,3)    {Flat};
			\node (opens) at (-8.5,-1.5)    {Open space};
			\node (rank) at (-2,0)    {Rank};
			\node (independents) at (4,0)   {Independence};
			\node (spanning) at (10,3)   {Spanning};
			\node (bases) at (4,3)   {Basis};
			\node (circuits) at (4,-3)  {Circuit};
			\node (bicolor) at (-2,-2)    {Bi-colourings};
			\node (hyperplanes) at (-4,-3)    {Hyperplane};
			\node (dependents) at (10,0)  {Dependence};
			\node (nonspanning) at (10,-3)  {Non-spanning};
			 
			\draw [open triangle 45-](independents) -- node[auto, rectangle, fill=blue!30, scale=0.7,anchor=center] {Cor. \ref{IndChius}}  (closure);
			\draw [open triangle 45-open triangle 45](dependents) -- node[auto, rectangle, fill=blue!30, scale=0.7,anchor=center] {Cor. \ref{DC}}  (circuits);
			\draw [open triangle 45-open triangle 45](closure)  -- node[auto, rectangle, fill=blue!30, scale=0.7,anchor=center] {Cor. \ref{cor:cl-flats-rank}}   (flats);		
			\draw [open triangle 45-open triangle 45](flats) -- node[auto, rectangle, fill=blue!30,scale=0.7,anchor=center] {Cor. \ref{OpenAxioms}}  (opens);
        	\draw [open triangle 45-open triangle 45](independents) -- node[auto, rectangle, fill=blue!30,scale=0.7,anchor=center] {Cor. \ref{Spanningq-Matr}}  (spanning);
		    \draw [open triangle 45-open triangle 45](dependents) -- node[auto, rectangle, fill=blue!30,scale=0.7,anchor=center] {Cor. \ref{NOSpanningq-Matr}}  (nonspanning);
			\draw [open triangle 45-open triangle 45](dependents) -- node[auto, rectangle, fill=blue!30, scale=0.7,anchor=center] {Cor. \ref{cor:ind-dep}}  (independents);
			\draw [open triangle 45-open triangle 45](hyperplanes) -- node[auto, rectangle, fill=blue!30, scale=0.7,anchor=center] {Cor. \ref{corr:flat-hyper}}  (flats);
			\draw [open triangle 45-open triangle 45](hyperplanes) -- node[auto, rectangle, fill=blue!30, scale=0.7,anchor=center] {Cor. \ref{HypCircq-Matr}}  (circuits);
			\draw [open triangle 45-open triangle 45](rank) -- node[auto, rectangle, fill=blue!30, scale=0.7, anchor=center] {\cite[Thm. 3]{WINEpaper1}}  (flats);
			\draw [open triangle 45-open triangle 45](rank) -- node[auto, rectangle, fill=blue!30, scale=0.7,anchor=center] {\cite[Thm. 8 ]{JP18}}  (independents);
			\draw [open triangle 45-open triangle 45](bases) -- node[auto, rectangle, fill=blue!30, scale=0.7,anchor=center] {\cite[Thm. 37 ]{JP18}}  (independents);
			\draw [open triangle 45-open triangle 45](rank) -- node[auto, rectangle, fill=blue!30, scale=0.7,anchor=center] { 
			Cor. \ref{cor:clrk}}  (closure);
			\draw [ -open triangle 45](independents) -- node[auto, rectangle, fill=blue!30, scale=0.7,anchor=center] {\cite[Thm. 64 ]{JP18}}  (circuits);
			\draw [ open triangle 45-open triangle 45](rank) -- node[auto, rectangle, fill=blue!30, scale=0.7,anchor=center] {\cite[Thm. 2 ]{BCJ17}}  (bicolor);
		}
		\caption{Cryptomorphisms}
		\label{diagram1} 
	\end{figure}

In Figure \ref{diagram1}, twelve different equivalent axiom systems of $q$-matroids are shown, which are $q$-analogues of classical axiomatic systems. As we show in this paper, these systems all equivalently define a $q$-matroid and hence form a family of $q$-cryptomorphisms.
These axioms are labelled as the {\em rank, closure, basis, independence, dependence, circuit, hyperplane, flat, open space, spanning space, non-spanning space}, and {\em bi-colouring axioms}.

Cryptomorphisms between the rank, independence and bases axioms were already proven in \cite{JP18}. For independence and bases, it turns out there is an extra axiom needed in addition to the classical case: simply taking a straightforward $q$-analogue of the classical axioms is sometimes insufficient to find axioms for a $q$-matroid. A cryptomorphism between the rank and flat axioms was shown in \cite{WINEpaper1} and that $q$-matroidal bi-colourings and the rank axioms equivalently define a $q$-matroid was shown in \cite{BCJ17}. That the rank axioms imply the closure axioms was shown in \cite{JP18}, while at that time it was not clear that those closure axioms were sufficient to define a $q$-matroid. We answer this question affirmatively by showing that the closure and independence axioms are cryptomorphic.

It was shown that the circuit axioms proposed in \cite{JP18} can be deduced from the independence axioms. Here we establish the converse by proving that the circuit and dependence axioms are cryptomorphic and that the dependence and independence axioms are cryptomorphic. However, we need a different axiom for the circuits than the one proposed in \cite{JP18}. This is again an illustration that taking straightforward $q$-analogues of classical axioms is sometimes insufficient. We see this problem also arising in the case of the open space axioms.
We furthermore show that the flat and hyperplane axioms are equivalent, from which we easily obtain cryptomorphisms with the open space axioms and the circuit axioms by dualization (and also via equivalence with the rank axioms).

In \cite{white}, various families are defined with respect to a given family of subsets, such as its {\em upper cone}, {\em lower cone}, {\em dual, opposite, max} and {\em min} families (see Definition \ref{def:families}). In Figure \ref{diagram2}, we illustrate the relations between the different axiom systems in the context of these notions. These follow exactly as for subsets, although for the dual of a family, we take the orthogonal complement with respect to an inner product. Another difference to note is that for the left side of the diagram --- bases, independence and spanning --- four axioms are needed, contrary to the three axioms in the classical case. This difference between the classical case and the $q$-analogue does not appear for the other axiom systems in the diagram.

	\begin{figure}[ht]
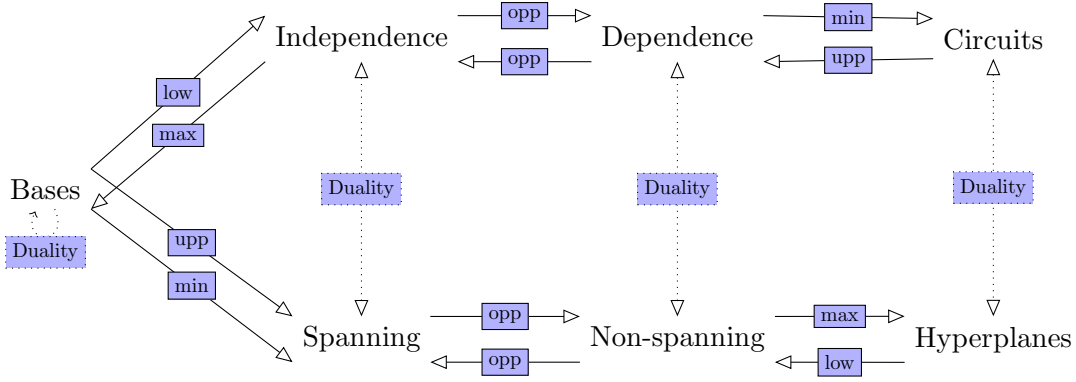

		\tikz[every node/.style={}, xscale=2.1]
		{
			\node (bases) at (0,0)  { Bases };
		 	\node (independents) at (2,2)  { Independence };
			 \node (spanning) at (2,-2)  { Spanning };
		 	\node (dependents) at (4,2)  { Dependence };
			 \node (nonspanning) at (4,-2)  { Non-spanning };
			 \node (circuits) at (6,2)  { Circuits };
			 \node (hyperplanes) at (6,-2)  { Hyperplanes };	
			 
			 \draw[dotted, -open triangle 45] (bases) [loop left];
			\draw [ -open triangle 45](bases.north east) -- node[auto, rectangle, draw,fill=blue!30, scale=0.7,anchor=center] { low }  (independents.north west);
			\draw [ -open triangle 45](independents.south west) -- node[auto, rectangle, draw,fill=blue!30, scale=0.7,anchor=center] { max }  (bases.south east);
			\draw [ -open triangle 45](independents.north east) -- node[auto, rectangle, draw,fill=blue!30, scale=0.7,anchor=center] { opp }  (dependents.north west);	
					\draw [ -open triangle 45](dependents.south west) -- node[auto, rectangle, draw,fill=blue!30, scale=0.7,anchor=center] { opp } (independents.south east);	
				\draw [ -open triangle 45](dependents.north east) -- node[auto, rectangle, draw,fill=blue!30, scale=0.7,anchor=center] { min }  (circuits.north west);	
					\draw [ -open triangle 45](circuits.south west) -- node[auto, rectangle, draw,fill=blue!30, scale=0.7,anchor=center] { upp }  (dependents.south east);	
			\draw [ -open triangle 45](bases.north east) -- node[auto, rectangle, draw,fill=blue!30, scale=0.7,anchor=center] { upp }  (spanning.north west);
			\draw [ -open triangle 45](bases.south east) -- node[auto, rectangle, draw,fill=blue!30, scale=0.7,anchor=center] { min }  (spanning.south west);
				\draw [ -open triangle 45](spanning.north east) -- node[auto, rectangle, draw,fill=blue!30, scale=0.7,anchor=center] { opp } (nonspanning.north west);	
					\draw [ -open triangle 45](nonspanning.south west) -- node[auto, rectangle, draw,fill=blue!30, scale=0.7,anchor=center] { opp }  (spanning.south east);		
			\draw [ -open triangle 45](nonspanning.north east) -- node[auto, rectangle, draw,fill=blue!30, scale=0.7,anchor=center] { max }  (hyperplanes.north west);	
					\draw [ -open triangle 45](hyperplanes.south west) -- node[auto, rectangle, draw,fill=blue!30, scale=0.7,anchor=center] { low }  (nonspanning.south east);	
\draw [dotted, open triangle 45-open triangle 45](independents.south) -- node[auto, rectangle, draw,fill=blue!30, scale=0.7,anchor=center] { Duality }  (spanning.north);
\draw [dotted, open triangle 45-open triangle 45](dependents.south) -- node[auto, rectangle, draw,fill=blue!30, scale=0.7,anchor=center] { Duality }  (nonspanning.north);
\draw [dotted, open triangle 45-open triangle 45](circuits.south) -- node[auto, rectangle, draw,fill=blue!30, scale=0.7,anchor=center] { Duality }  (hyperplanes.north);

\path 
(bases) edge [dotted,loop below] node[auto, rectangle, draw,fill=blue!30, scale=0.7] { Duality } (bases)

		}
		\caption{Cryptomorphisms with Duality}
		\label{diagram2} 
	\end{figure}

This paper is organised as follows. In Section \ref{sec:prelim} we outline all the different axiomatic systems that we will consider in this work. 
In Section \ref{sec:representable} we present an infinite family of representable $q$-matroids derived from an $\Fqm$-linear code and explicitly describe its bases, independent spaces, flats, circuits etc.
In Section \ref{sec:equivaxiomsyst} we describe some variations on the independence, hyperplane and rank axiom systems. 
The remaining sections go through the various pairwise cryptomorphisms in turn.
In Section \ref{sec:ind-cl} we show that the independence and closure axioms are cryptomorphic. 
In Sections \ref{sec:flat-cl} and \ref{sec:rkcl} we prove cryptomorphisms between the closure function axioms and the independence and rank function axioms, respectively.  
In Sections \ref{FlatHypCryptomorph} and \ref{OpensSect} we establish the equivalence of the flat axioms and the hyperplane and open space axioms
respectively. 
In Sections \ref{sec:ind-dep} and \ref{sec:dep-circ}, the dependence axioms are shown to be cryptomorphic to the independence and the circuit axioms, respectively. 
In Section \ref{sec-HC} we note the cryptomorphism between the hyperplanes axioms and circuit axioms and discuss cocircuits of a matroid.
Finally, in Section \ref{ToSpanOrNotToSpan}, we deduce the spanning and non-spanning space axioms from the other axiom systems.

 \section{Preliminaries}\label{sec:prelim}

 	Throughout this paper, $n$ denotes a fixed positive integer and we denote by $E$ a fixed $n$-dimensional vector space over an arbitrary field
 $\mathbb{F}.$ We denote by $\mathcal{L}(E)$ the lattice of subspaces of $E$. For any $A,B\in\mathcal{L}(E)$ with $A\subseteq B$ we denote by $[A,B]$ the interval between $A$ and $B$, that is, the lattice of all subspaces $X$ with $A\subseteq X\subseteq B$. For $A\subseteq E$ we use the notation $\mathcal{L}(A)$ to denote the interval $[\{0\},A]$.

 For any subspace $X \in \kL(E)$ we denote by $X^\perp$ the orthogonal complement of $X$ in $E$ with respect to the standard dot product:
 \[ X^\perp :=\{ y \in E : x \cdot y = 0 \; \forall x \in X \},\]
 where $x \cdot y := \sum_{j=1}^n x_i y_i$ for any $x=(x_1,\ldots,x_n),y=(y_1,\ldots,y_n) \in E$.

\begin{Definition}\label{def:families}
   Let $\A \subseteq \kL(E)$. We define the following families of subspaces of $E$.
   \begin{align*}
       \upp(\A)&:=\{ X \in \kL(E) : \exists A \in \A, A \subseteq X \},\\
       \low(\A)&:=\{ X \in \kL(E) : \exists A \in \A, X \subseteq A \},\\
       \max(\A)&:=\{ X \in \A : X \nsubseteq A \text{ for any } A \in \A, A \neq X \},\\
       \min(\A)&:=\{ X \in \A :  A \nsubseteq X \text{ for any } A \in \A, A \neq X \},\\
       \opp(\A)&:=\{ X \in \kL(E) : X \notin \A \},\\
       \A^\perp&:=\{ X^\perp : X \in \A \}.  
   \end{align*}
\end{Definition}

\begin{Definition}\label{def:maxinX}
	Let $\A$ be a collection of subspaces of $E$. For any subspace $X \in \kL(E)$, we define the collection of {\bf maximal subspaces
	of $X$ in $\A$} to be the collection of subspaces
	\[
	\max(X,\A):=\{ A \in \A : A \subseteq X \text{ and } B \subset X, B \in \A \implies \dim(B) \leq \dim(A) \}.
	\]
	In other words, $\max(X,\A)$ is the set of subspaces of $X$ in $\A$ that have maximal dimension over all such choices of subspaces.
\end{Definition}

The following defines a $q$-matroid in terms of a rank function (see \cite{JP18}).

 \begin{Definition}\label{rankfunction}
	A $q$-matroid $M$ is a pair $(E,r)$ where
	$r$ is an integer-valued 
	function defined on the subspaces of $E$ with the following properties:
	\begin{itemize}
		\item[(R1)] For every subspace $A\in \kL(E)$, $0\leq r(A) \leq \dim A$. 
		\item[(R2)] For all subspaces $A\subseteq B \in \kL(E)$, $r(A)\leq r(B)$. 
		\item[(R3)] For all $A,B$, $r(A+ B)+r(A\cap B)\leq r(A)+r(B)$.  
	\end{itemize}
	The function $r$ is called the {\bf rank function} of the $q$-matroid. 
\end{Definition}

 	\begin{Definition}\label{independentspaces}
		Let $(E,r)$ be a $q$-matroid. 
		A subspace $A$ of $E$ is called an {\bf independent} space of $(E,r)$ if $$r(A)=\dim A.$$
		We write $\I_r$ to denote the set of independent spaces of the $q$-matroid $(E,r)$:
		\[\I_r :=\{ I \in \kL(E) : \dim(I) = r(I) \}. \]
		
		A subspace that is not an independent space of $(E,r)$ is called a 
		{\bf dependent space} of the $q$-matroid $(E,r)$. 
		We call $C \in \kL(E)$ a {\bf circuit} if it is itself a dependent space and every proper subspace of $C$ is independent.
		A {\bf spanning space} of the $q$-matroid $(E,r)$ is a subspace $S$  such that $r(S)=r(E)$. 
		A {\bf non-spanning space} of the $q$-matroid $(E,r)$ is a space that is not a spanning space.
		We write $\kS_r$ to denote the set of spanning spaces of $(E,r)$ and we write $\kN_r$ to denote its set of non-spanning spaces. 
		A subspace is called an {\bf open space} of $(E,r)$ if it is a (vector space) sum of circuits.
	    We write $\kO_r$ to denote the set of open spaces of $(E,r)$.
	\end{Definition}

    \begin{Definition}\label{def-closure}
        Let $(E,r)$ be a $q$-matroid.
        For each $A \in \kL(E)$, define $C_r(A):=\{x \in \kL(E): r(A+x)=r(A), \dim(x)=1 \}.$
    	The {\bf closure function} of a $q$-matroid $(E,r)$ is the function
    	defined by
    	\[\cl_r:\mathcal{L}(E) \to\mathcal{L}(E): A \mapsto \cl_r(A)=\sum_{x \in C_r(A)} x .
    	\]
    \end{Definition}
   
    \begin{Definition}\label{def:flat}
	A subspace $A$ of a $q$-matroid $(E,r)$ is called a {\bf flat} if for all $1$-dimensional 
	subspaces $x \in \kL(E)$ such that $x\nsubseteq A$ we have $$r(A+x)>r(A).$$
	We write ${\mathcal F}_r$ to denote the set of flats of the $q$-matroid $(E,r)$, that is
	\[ \kF_r:= \{ A \in \kL(E) : r(A+x)>r(A) \:\:\forall x \in \kL( E), x \nsubseteq A, \dim(x)=1  \}. \]
	A subspace $H$ is called a $\textbf{hyperplane}$ if it is a maximal proper flat, i.e., if $H \neq E$ and the only flat that properly contains $H$ is $E$. We write $\mathcal{H}_r$ to denote the set of hyperplanes of the $q$-matroid $(E,r)$, that is
	\[ \kH_r=\{A\in\kL(E): r(A)=r(M)-1 \text{ and } r(A+x)>r(A) \:\:\forall x \in \kL( E), x \nsubseteq A, \dim(x)=1  \}. \] 
\end{Definition}

    We now present several axiom systems. Some of these, such as the {\em independence axioms, flat axioms, circuit axioms} and
    {\em closure axioms} have been presented before, while others (namely the axioms of {\em open spaces} and {\em dependent spaces}) have not. In later sections we will establish that these are all cryptomorphisms
    of a $q$-matroid.

	\begin{Definition}\label{independence-axioms}
	Let $\I \subseteq \mathcal{L}(E)$. We define the following {\bf independence axioms}.
	\begin{itemize}
		\item[(I1)] $\I\neq\emptyset$.
		\item[(I2)] For all $I,J \in \kL(E)$, if $J\in\I$ and $I\subseteq J$, then $I\in\I$.
		\item[(I3)] For all $I,J\in\I$ satisfying $\dim I<\dim J$, there exists a $1$-dimensional subspace $x\subseteq J$, $x\not\subseteq I$ such that $I+x\in\I$.
		\item[(I4)] For all $A,B \in \kL(E)$ and $I,J \in \kL(E)$ such that 
		$I \in \max(\I \cap \kL(A))$ and $J \in \max(\I \cap \kL(B))$,
		there exists $K\in \max(\I \cap \kL(A+B))$ such that $K \subseteq I+J$.
	\end{itemize}
	If $\I$ satisfies the independence axioms (I1)-(I4) we say that $(E,\I)$ is a collection of {\bf independent spaces}.
\end{Definition}

\begin{Definition}\label{indep-bases}
Let $\mathcal{B} \subseteq \kL(E)$.
We define the following {\bf basis axioms}.
\begin{itemize}
\item[(B1)] $\mathcal{B}\neq\emptyset$
\item[(B2)] For all $B_1,B_2\in\mathcal{B}$, if $B_1\subseteq B_2$ then $B_1=B_2$.
\item[(B3)] For all $B_1,B_2\in\mathcal{B}$ and for every subspace $A$ of codimension 1 in $B_1$ satisfying $B_1\cap B_2\subseteq A$, there is a $1$-dimensional subspace $y$ of $B_2$ such that $A+y\in\mathcal{B}$.
\item[(B4)] 
For all $A,B\in \kL(E)$, if $I$ and $J$ are maximal intersections of some members of $\mathcal{B}$ with $A$ and $B$, respectively, there exists a maximal intersection of a basis and $A+B$ that is contained in $I+J$.
\end{itemize}
If $\mathcal{B}$ satisfies the bases axioms (B1)-(B4) we say that $(E,\mathcal{B})$ is a collection of {\bf bases}.
\end{Definition}

\begin{Definition}
    Let $\A \subseteq \kL(E)$. Let $A,B \in \A$. We say that $B$ {\bf covers} $A$ in $\A$
    if $A \subseteq B$ and for any $C \in \A$ such that
    $A \subseteq C \subseteq B$ then either $A=C$ or $B=C$.
\end{Definition}

\begin{Definition}\label{flat}
Let $\mathcal{F} \subseteq {\mathcal L}(E)$. 
We define the following {\bf flat axioms}.
\begin{itemize}
  	\item[(F1)] $E\in\mathcal{F}$.
  	\item[(F2)] If $F_1\in\mathcal{F}$ and $F_2\in\mathcal{F}$, then $F_1\cap F_2\in\mathcal{F}$.
  	\item[(F3)] For all $F\in\mathcal{F}$ and $x\in \kL(E)$ a $1$-dimensional subspace not contained in $F$, there is a unique cover of $F$ in $\kF$ that contains $x$.
\end{itemize}
If $\kF$ satisfies the flat axioms (F1)-(F3) we say that $(E,\kF)$ is a collection of {\bf flats}.
\end{Definition}

\begin{Definition}\label{OpenSpaces}
Let $\kO \subseteq \kL(E)$. 
We define the following {\bf open space axioms}.
\begin{itemize}
 \item[(O1)] $\{0\} \in \kO$.
 \item[(O2)] For all $O_1,O_2 \in \kO$ it holds that $O_1 +O_2 \in \kO$.
 \item[(O3)] For each $O \in \kO$ and each $X \in \kL(E)$ such that $O \nsubseteq X$ and $\codim_E(X)=1$, there exists a unique $O' \subseteq X \cap O$ such that $O$ is a cover of $O'$ in $\kO$.
\end{itemize}
If $\kO$ satisfies the open space axioms (O1)-(O3) we say that $(E,\kO)$ is a collection of {\bf open spaces}.

\end{Definition}
 
 \begin{Definition}\label{hyperplane-axioms}
Let $\mathcal{H} \subseteq \mathcal{L}(E)$. 
We define the following {\bf hyperplane axioms}.
\begin{itemize}
  	\item[(H1)] $E\notin\mathcal{H}$.
  	\item[(H2)] For all $H_1,H_2\in\mathcal{H}$, if $H_1\subseteq H_2$ then $H_1=H_2$.
  	\item[(H3)] For all distinct $H_1,H_2\in\mathcal{H}$, for each 1-dimensional space $x\in \kL(E)$ there exists $H_3\in\mathcal{H}$ such that $(H_1\cap H_2)+x\subseteq H_3$.
\end{itemize}
If $\kH$ satisfies the axioms (H1)-(H3) then we say that $(E,\kH)$ is a collection of {\bf hyperplanes}.
\end{Definition}

\begin{Definition}\label{dependence-axioms}
	Let $\mathcal{D}\subseteq\mathcal{L}(E)$. We define the following {\bf dependence axioms}.
	\begin{itemize}
		\item[(D1)] $\{0\}\notin \mathcal{D}$.
		\item[(D2)] For all $D_1,D_2 \in \kL(E)$, if $D_1\in\mathcal{D}$ and $D_1\subseteq D_2$ then $D_2 \in \mathcal{D}$.
		\item[(D3)]  For all $D_1,D_2 \in \mathcal{D}$ satisfying $D_1 \cap D_2 \notin \mathcal{D}$,
		if $D$ is a space of codimension one in $D_1+D_2$ then $D \in \kD$.
		\end{itemize}
	If $\mathcal{D}$ satisfies the dependence axioms (D1)-(D3) we say that
	$(E,\mathcal{D})$ is a collection of {\bf dependent spaces}.
\end{Definition}

\begin{Definition}\label{circuit-axioms}
Let $\mathcal{C}\subseteq\mathcal{L}(E)$. We
define the following {\bf circuit axioms}.
\begin{itemize}
\item[(C1)] $\{0\}\notin\mathcal{C}$.
\item[(C2)] For all $C_1,C_2\in\mathcal{C}$, if $C_1\subseteq C_2$ then $C_1=C_2$.
\item[(C3)]  For distinct $C_1,C_2 \in \kC$ and any $X\in \kL(E)$ of codimension $1$ there is a circuit $C_3 \subseteq \kC$ such that $C_3 \subseteq (C_1+C_2)\cap X$.
\end{itemize}
If $\kC$ satisfies the circuit axioms (C1)-(C3), we say that $(E,\mathcal{C})$ is a collection of {\bf circuits}.
\end{Definition}
Note that the axiom (C3) listed here is different from the axiom (C3) as defined in \cite[Theorem 64]{JP18}. We will explain this in Section \ref{sec-HC}.

\begin{Definition}\label{closure-axioms}
Let $\cl:\mathcal{L}(E)\to \mathcal{L}(E)$ be a map. We define the following {\bf closure axioms}.
\begin{itemize}
\item[(Cl1)] For every $A \in \kL(E)$ it holds that $A\subseteq\cl(A)$.
\item[(Cl2)] For all $A,B \in \kL(E)$, if $A\subseteq B$ then $\cl(A)\subseteq\cl(B)$.
\item[(Cl3)] For every $A \in \kL(E)$ it holds that $\cl(A)=\cl(\cl(A))$.
\item[(Cl4)] For all $x,y,A \in \kL(E)$ such that $\dim(x)=\dim(y)=1$, if $y\subseteq\cl(A+x)$ and $y\not\subseteq\cl(A)$, then $x\subseteq\cl(A+y)$.
\end{itemize}
If $\cl:\mathcal{L}(E)\to \mathcal{L}(E)$ satisfies the closure axioms (Cl1)-(Cl4) then we call it a {\bf closure function}.
We write $(E,\cl)$ to denote a vector space $E$ together with a function $\cl$ satisfying the closure axioms.
\end{Definition} 	

\begin{Definition}\label{spanning-axioms}
	Let $\kS \subseteq \mathcal{L}(E)$. We define the following {\bf spanning space axioms}.
	\begin{itemize}
		\item[(S1)] $E \in \kS$.
		\item[(S2)] For all $I,J \in \kL(E)$, if $J\in\kS$ and $J \subseteq I$ then $I\in\kS$.
		\item[(S3)] For all $I,J\in\kS$ such that $\dim J<\dim I$, there exists some $X \in \kL(E)$ of codimension $1$ in $E$ satisfying $J \subseteq X$, $I \nsubseteq X$, and $I\cap X\in\kS$.
		\item[(S4)] For all $A,B\in\kL(E)$ and $I,J\in \kL(E)$ such that $I\in\min(\kS\cap[A,E])$ and $J\in\min(\kS\cap[B,E])$, there exists $K\in\min(\kS\cap [A\cap B,E])$ such that $I\cap J\subseteq K$.		
	\end{itemize}
	If $\kS$ satisfies the independence axioms (S1)-(S4) we say that $(E,\kS)$ is a collection of {\bf spanning spaces}.
\end{Definition}

\begin{Definition}\label{nonspan-axioms}
	Let $\mathcal{N}\subseteq\mathcal{L}(E)$. We define the following {\bf non-spanning space axioms}.
	\begin{itemize}
		\item[(N1)] $E \notin  \mathcal{N} $.
		\item[(N2)] For all $N_1,N_2 \in \kL(E)$, if $N_1\in\mathcal{N}$ and $N_2\subseteq N_1$ then $N_2 \in \mathcal{N}$.
		\item[(N3)]  For all $N_1,N_2 \in \mathcal{N}$ satisfying $N_1 +N _2 \notin \mathcal{N}$, if $N$ is a space such that $N_1\cap N_2$ has codimension one in $N$ then $N\in\kN$.
		\end{itemize}
	If $\mathcal{N}$ satisfies the dependence axioms (N1)-(N3) we say that
	$(E,\mathcal{N})$ is a collection of {\bf non-spanning spaces}.
\end{Definition}

We conclude this section with the notion of a {\em dual matroid}, which we will use in Sections \ref{sec-HC}, \ref{OpensSect} and \ref{ToSpanOrNotToSpan}.

\begin{Definition}
Let $M=(E,r)$ be a $q$-matroid. Then $M^*=(E,r^*)$ is also a $q$-matroid, called the \textbf{dual $q$-matroid}, with rank function
\[ r^*(A)=\dim(A)-r(E)+r(A^\perp). \]
\end{Definition}

We recall the following theorem from \cite[Theorem 45]{JP18}.

\begin{Theorem}\label{thm-dualbases}
The subspace $B\in \kL(E)$ is a basis of the $q$-matroid $M$ if and only if $B^\perp$ is a basis of the dual $q$-matroid $M^*$.
\end{Theorem}

\section{An Infinite Family of Representable $q$-Matroids}\label{sec:representable}

We present a construction of an infinite family of $q$-matroids. For a specific choice of parameter sets, we will identify its independent and dependent spaces, spanning and non-spanning spaces, its circuits, hyperplanes, open spaces, bases, flats and characterize the rank and closure of each subspace.

We first recall a standard construction of a representable $q$-matroid over a finite field (see \cite{JP18}). Let $E = {\Fq^n}$ and let $k,m$ be positive integers.
Let $h:\FF_{q^m}^n \longrightarrow \FF_{q^m}^k$ be an $\FF_{q^m}$-epimorphism. 
We define the function
$$r : \kL(E) \longrightarrow \mathbb{N}_0 : A \mapsto r(A) :=\dim_{\FF_{q^m}}(h(A)).$$
Then $(E,r)$ is a $q$-matroid with rank function $r$; the rank of a subspace $A$ is the dimension of its image under the epimorphism $h$.
We have $r(E) = k$. The epimorphism $h$ can be equivalently represented by a matrix $G$ with respect to some choice of basis for $\FF_{q^m}^n$ and 
$\FF_{q^m}^k$, while for each subspace $A \in \kL(E)$, we have that $r(A)$ is the $\FF_{q^m}$-rank of the matrix product $GY$ for any matrix $Y$ whose columns form a basis of $A$. We will denote this $q$-matroid by $M[G]$.

As a preparation for our construction, we describe the following setting.
Let $m=ps$ for coprime integers $p$ and $s$ and let $\alpha$ be a primitive element of $\FF_{q^m}$. 
Define $e:=\frac{q^m-1}{q^s-1}$, so that $\alpha^e$ has order $q^s-1$ in $\FF_{q^m}^\times $ and in particular is a generator of the subfield $\FF_{q^s}$.
Consider the $\Fq$-spaces 
$G_i=\langle \alpha^i,\alpha^{i+e},\ldots,\alpha^{i+(s-1)e} \rangle \subseteq \FF_{q^m},$
defined for $1\leq i \leq e$.
There exist $f_j \in \Fq$ such that $\displaystyle \sum_{j=0}^{s-1} f_j \alpha^{i+je}=0$, if and only if $\alpha^e$ is a root of a polynomial of degree at most $s-1$.
This is clearly impossible, since $\alpha^e$ is a primitive element of $\FF_{q^s}$, and so its minimal polynomial over $\Fq$ has degree $s$.
It follows that $G_i$ has $\Fq$-dimension equal to $s$. 
Moreover, the spaces $G_i$ have trivial intersection. Indeed, for $1\leq i,j \leq e$, there exist $f_k, g_k \in \Fq$ satisfying
$\displaystyle \sum_{k=0}^{s-1}f_k\alpha^{i+ek} =\sum_{k=0}^{s-1}g_k\alpha^{j+ek}$ if and only if $\displaystyle \alpha^{i-j} = \frac{g(\alpha^e)}{f(\alpha^e)}$ for some polynomials
$f(x),g(x) \in \FF_q[x]$.
This holds only if $\alpha^{i-j} \in \FF_{q^s}$, which holds if and only if $(i-j)(q^s-1) \equiv 0 \mod q^m-1$, in which case we must have $i=j$.
Therefore, the collection of spaces $\kG:=\{G_i: 1\leq i \leq e\}$ form a spread in $\FF_{q^m}$.
In fact $\kG$ is a {\em Desarguesian spread} and this construction is well-known \cite{segre}. 
We will use $\kG$ to characterise the ranks of spaces associated with an infinite family of representable $q$-matroids. Before we 
characterise this family, we will consider a particular example.

	\begin{example}\label{M2s}
		Let $s\in \NN$ be an odd integer and let $m=2s$. Let $\alpha \in \FF_{q^m}$ a primitive element.
		Take as a basis for $\FF_{q^m}$ over $\FF_q$ the elements $1,\alpha, \alpha^2,\ldots,\alpha^{2s-1}$ and consider the 
		matrix
		\[
		G=\left[\begin{matrix}1 &\alpha & \alpha^2 &\ldots &\alpha^{2s-1}\\ 
			1 &\alpha^{q^s} & (\alpha^{q^s})^2 &\ldots &(\alpha^{q^s})^{2s-1}
		\end{matrix} \right].
	    \]
		As outlined above, $G$ determines a $q$-matroid $(\Fq^n,r)$, which clearly supports only the possible ranks $0,1,2$, as $G$ itself has rank $2$, so in particular, $r(\FF_{q^m}^n)=2$. As $G$ has no all-zero columns, every $1$-dimensional space of $\FF_{q^m}$ over $\FF_q$ has rank $1$.
		Let $e = \frac{q^m-1}{q^s-1}=q^s+1$. The collection of $s$-dimensional subspaces
		$G_i=\langle \alpha^i, \alpha^{e+i},\ldots,\alpha^{(s-1)e+i}  \rangle_{\FF_q},$
		for $1 \leq i \leq e$ form a spread of $\FF_{q^m}$ as a vector space over $\Fq$.
		As will be shown in Theorem \ref{th:matex}, $r(G_i)=1$ for each $i$, while every other $s$-dimensional space
		has rank 2.
		Let us specify our example in a very small case. For $m=6,q=2$ we get a $q$-matroid $M_6$ with ground space $\FF_{2^6}$ over $\FF_2$.
		The spread $\kG$ is a collection of $e=2^3+1=9$ spaces of $\FF_2$-dimension $3$ and rank $1$, which we denote by $G_1,\ldots,G_9$. We list these as the following binary vector spaces.
		\begin{align*}
		    G_1 & =\langle (0,1,0,0,0,0),(0,0,0,0,1,1),(0,1,1,1,1,0) \rangle_{\FF_2}, \\
		    G_2 & =\langle (0,0,1,0,0,0),(1,1,0,0,0,1),(0,0,1,1,1,1) \rangle_{\FF_2}, \\
		    G_3 & =\langle (0,0,0,1,0,0),(1,0,1,0,0,0),(1,1,0,1,1,1) \rangle_{\FF_2}, \\
		    G_4 & =\langle (0,0,0,0,1,0),(0,1,0,1,0,0),(1,0,1,0,1,1) \rangle_{\FF_2}, \\
		    G_5 & =\langle (0,0,0,0,0,1),(0,0,1,0,1,0),(1,0,0,1,0,1) \rangle_{\FF_2}, \\
		    G_6 & =\langle (1,1,0,0,0,0),(0,0,0,1,0,1),(1,0,0,0,1,0) \rangle_{\FF_2}, \\
		    G_7 & =\langle (0,1,1,0,0,0),(1,1,0,0,1,0),(0,1,0,0,0,1) \rangle_{\FF_2}, \\
		    G_8 & =\langle (0,0,1,1,0,0),(0,1,1,0,0,1),(1,1,1,0,0,0) \rangle_{\FF_2}, \\
		    G_9 & =\langle (0,0,0,1,1,0),(1,1,1,1,0,0),(0,1,1,1,0,0) \rangle_{\FF_2}. \\
		\end{align*}
		
		Each space $G_i$ contains $7$ distinct $2$-dimensional spaces and no space is contained in two spread elements, so in total
		we have $63$ $2$-dimensional spaces contained some $G_i$, which we denote by $D_1,\ldots,D_{63}$.
		Clearly $r(D_i) =  1$ for each $i\in \{1,\ldots,63\}$. 
	
		In Table \ref{Tabellona}, we tabulate how the subspaces of each dimension in $\FF_2^6$ are distributed, according to the different cryptomorphic definitions of a $q$-matroid. 
		The closure function, independent spaces, circuits etc are all defined with respect to the given rank function.
		
			\begin{table}
			\small
		\resizebox{\textwidth}{!}{	\begin{tabular}{|c|c|c|c|c|c|c|c|}
				\hline
				\diagbox{\thead{crypt}}{\thead{dim}} &\thead{0} & \thead{1} & \thead{2} & \thead{3} & \thead{4} & \thead{5} & \thead{6}  \\\hline
			
			{\bf Rank} & 0 & 1& \makecell{2 except\\ $r(D_1)=\cdots =r(D_{63})=1$} & \makecell{2 except\\ $r(G_1)=\cdots=r(G_9)=1$} &2 &2 &2\\ \hline
			{\bf \makecell{Closure \\ of a Space}} &0 
			& \makecell{ $\cl(x)=G_i$ \\ for $x \subseteq G_i$} & 
			$\cl(T)=\left\{ \begin{array}{ll} G_i & \text{ if } T \subseteq G_i\\ E & \text{ else} \end{array}\right.$ 
			 & $\cl(T)=\left\{ \begin{array}{ll} G_i & \text{ if } T = G_i\\
			 E & \text{ else }\end{array}\right.$& $E$ &$E$ &$E$
			\\ \hline
	{\bf \makecell{Independent \\Spaces}} & yes & all & \makecell{all except \\$D_1,\ldots,D_{63}$} & none & none & none & no 
			\\ \hline
		 {\bf Bases} & no & none & \makecell{all except \\$D_1,\ldots,D_{63}$} & none & none & none & no 
			\\ \hline
			 {\bf \makecell{Spanning \\ Spaces}} & no & none & \makecell{all except \\$D_1,\ldots,D_{63}$} & \makecell{all except\\ $G_1,\ldots,G_9$} & all & all & yes
			\\ \hline
		{\bf Circuits} & no & none & $D_1,\ldots,D_{63}$ & \makecell{ $T$ such that $D_i \nsubseteq T$ }& none & none & no 
			\\ \hline
			{\bf \makecell{Dependent \\Spaces}} & no & none & $D_1,\ldots,D_{63}$ &  all & all & all & yes 
			\\ \hline
			 {\bf \makecell{Non-spanning \\ Spaces}} & yes & all & $D_1,...,D_{63}$ & $G_1,...,G_9$ & none & none & no
			\\ \hline
	  {\bf Flats} & yes & none & none &  $G_1,\ldots,G_9$ & none & none & yes 
			\\ \hline		
	  {\bf Open Spaces} & yes & none & $D_1,\ldots,D_{63}$ &  \makecell{$G_1,\ldots,G_9$ and \\ $T$ such that $D_i \nsubseteq T$} & all & all & yes 
			\\ \hline	
{\bf Hyperplanes} & no & none & none &  $G_1,\ldots,G_9$ & none & none & no 
			\\ \hline
			\end{tabular}
		}	\caption{Defining Spaces of the $q$-Matroid.}\label{Tabellona}
		\end{table}
		
        As can be seen in Table \ref{Tabellona},
        every space of dimension at most $1$ has rank equal to its dimension, and so is independent.
		The zero space is also a flat, being equal to its closure, and is also a non-spanning space.
		 The closure of a one dimensional space is exactly one space from among the $G_1,\ldots,G_9$, namely the specific spread element $G_i$ that contains it. 
		 
		 As regards the spaces of dimension $2$, they all have rank 2 and are independent, bases and spanning spaces, except for the 63 subspaces of the spread,  $D_1,\ldots,D_{63}$, which are circuits and so are dependent, non-spanning, and open spaces. 
		 
		 Every $3$-dimensional space is dependent, having dimension exceeding its rank. In particular, as noted before, each $G_i$ has rank 1, while the remaining $3$-spaces have rank 2. Among the 1395 spaces of dimension 3, 1332 are circuits except those 63 spaces that contain some $D_i$ as a subspace. All spaces apart from $G_1,\ldots,G_9$ are  spanning spaces. The spaces $G_1,\ldots,G_9$ are flats, non-spanning and are also the only hyperplanes of $M_6$. 
		 Any open space of dimension 3, begin a sum of circuits, is either a circuit of dimension $3$ or has the form $D_i+D_j$, which must therefore be a spread element since any two $D_i,D_j$ are either contained in the same spread element, or have trivial intersection.
		 
		The $4$- and $5$-dimensional spaces are all dependent of rank 2 and there are no circuits nor flats among them. They are all spanning spaces. 
		The $4$-dimensional open spaces are the sum of open spaces of dimension $2$ and $3$. Each one contain some $D_i$ since every $4$-space intersects some spread element in dimension at least $2$, so all $4$-dimensional sets are open.
		The five dimensional spaces are also all open, because they are sums of open spaces of dimension $2$ and $3$. 
		Finally, the whole ground space is a dependent space of rank 2 and is not a circuit, but is a flat, a spanning space and an open space. \\
		
		We will now illustrate the multiple axiom systems for this example. Some axioms are straightforward to check directly for all possibilities, but we do not go through all the details. In other cases we pick some of the more illuminating examples.
	
	    \smallskip
		\noindent{\bf Rank:} (R1) and (R2) clearly hold. Let us see an example for (R3), using $G_1$ and $G_2$. We know that $G_1+G_2=E $ and $G_1 \cap G_2 =\{0\}$. Therefore,
			$r(G_1+G_2)+r(G_1 \cap G_2)=2+0=2 \leq r(G_1)+r(G_2)=1+1=2. $

    \smallskip
	\noindent{\bf Closure:} That axioms (Cl1)-(Cl3) hold is immediate, as we can see from the table shown in Table \ref{Tabellona}. 
	We'll show that (Cl4) holds. Let $A,x,y$ be subspaces of $\FF_2^6$ such that
	$\dim(x)=\dim(y)=1$. Suppose that $y \subseteq \cl(x+A)$ and that $y \nsubseteq \cl(A)$.
	As observed in Table \ref{Tabellona}, for any subspace $T$ we have $\cl(T) = E$ unless $T$ is contained in a spread element $G$, in which case we have $\cl(T)=G$. 
	Therefore, since $y \nsubseteq \cl(A)$, $y$ and $A$ are not both contained in the same spread element and hence $\cl(y+A)=E$.
	It follows that $x \subseteq \cl(y+A)$.
	
	\smallskip
	\noindent{\bf Independence:} It is clear from Table \ref{Tabellona} that (I1) and (I2) hold. 
	We'll show that (I3) holds. All the independent spaces determined by the rank function of $M_6$ have dimension at most $2$, so we need only consider some 1-dimensional subspace $I$ and and a two-dimensional space $J$, different from $D_1,\ldots,D_{63}$. Since $J \neq D_i$ for any $i$, it is not contained in any spread element. In particular, $J \nsubseteq G$ where $G$ is the unique spread element containing $I$.
	Therefore, there exists a 1-dimensional space $x \subseteq J$, $x \nsubseteq G$ and 
	$x+I$ is a $2$-dimensional space not contained in $G$, which is therefore independent.
	
	Consider now (I4). 
	Let $A,B$ be subspaces of $\FF_2^6$ and let $I,J$ be maximal independent subspaces of $A$ and $B$, respectively. Then $r(A)=\dim(I)$ and $r(B)=\dim(J)$. 
	Any independent subspace of $A+B$ has dimension at most 2. 
	If $\dim(I)=\dim(J)=1$ then $r(A)=r(B)=1$, so from Table \ref{Tabellona}, 
	$A$ and $B$ are each contained in some spread element.

	We have $\dim(I+J) =2$ and further, $I+J$ is independent if and only if $I+J \neq D_{\ell}$ for any $\ell$.
	If $A$ and $B$ are contained in distinct spread elements then 
	$I+J\neq D_{\ell}$ for any $\ell$ and so $I+J$ gives the required maximal independent subspace of $A+B$.
	If $A,B \subseteq G$ for a spread element $G$ then $r(A+B)=1$ and both $I$ and $J$ are a maximal 
	independent subspaces of $A+B$.
	If $\dim(I) =2$, then $I$ is a maximal independent subspace of $A+B$. This proves that (I4) holds for the independent spaces of the $q$-matroid $M_6$.

   \smallskip
   \noindent{\bf Bases:} That (B1) and (B2) hold is easy to see. 
   Let $B_1\neq B_2$ be a pair of distinct bases of the $q$-matroid $M_6$. 
   Then the $B_i$ are 2-dimensional spaces different from $D_1,\ldots,D_{63}$.
   Let $I=B_1\cap B_2$. 
   If $\dim(I)=1$ then $I$ is the only space of codimension $1$ in $B_1$ that contains $I$,
   so set $A=I$. 
   Otherwise, let $A$ be any $1$-dimensional space in $B_1$.
   In order to find a basis and see that (B3) holds, it is enough to add any $1$-dimensional space not contained in the same spread element as $A$.
   
   We illustrate an instance of (B4). Let 
   $A=\langle (1,0,0,1,0,0)\rangle$ and let \\
   $B=\langle (1,0,0,1,0,0), (1,0,0,0,0,1), (1,0,0,0,0,0) \rangle$. The maximal intersection of $A$ with a basis is $I=A$, while the maximal intersection of $B$ with a basis is $J=\langle (1,0,0,1,0,0), (1,0,0,0,0,1) \rangle$ ($J$ is a basis since it is a space of dimension $2$ not contained in a spread element). 
   Then $I+J = J$, which gives the required maximal intersection of a basis with $A+B = I+B = B$.
			
	\smallskip
	\noindent{\bf Circuits:} The axioms (C1) and (C2) are trivially satisfied. We'll show an example of 
	the axiom (C3). Let $C_1=\langle (0,0,0,0,0,1),$ $ (0,0,1,0,1,0) \rangle$ and $C_2= \langle (0,0,0,0,0,1), (1,0,0,1,0,0) \rangle$. Let $H=\langle (0,0,0,0,0,1) \rangle^\perp$. 
	Then $(C_1+C_2)\cap H$ contains the circuit, $C_3=\langle (0,0,1,0,1,0), (1,0,0,1,0,0) \rangle$, as required.
			
	\smallskip
	\noindent{\bf Dependence:}	(D1) and (D2) hold trivially. To illustrate (D3), take for example, two dependent spaces $D_i,D_j$, $1 \leq i,j \leq 63$. Being circuits, their intersection is independent. If $D_i,D_j$ come from the same spread element $G_l$, then their sum is $G_l$ itself and any codimension one space in such a sum is dependent. If they come from two different spread elements $G_l,G_m$, their intersection is $\{0\}$, which is independent. Their sum has dimension $4$ and hence any subspace of codimension 1 in $D_i+D_j$ is dependent.
		
	\smallskip
	\noindent{\bf Flats:} From Table \ref{Tabellona}, $E$ is a flat (so (F1) holds), $\{0\}$ is a flat and $G_1,\ldots,G_9$ are flats. This makes (F2) direct: the intersection of a flat $F$ with $E$ is $F$ itself, the intersection with $\{0\}$ or between two spread elements is $\{0\}$. As regards (F3), if $F=\{0\}$ and we take a $1$-dimensional space $x$, the unique cover of $F+x=x$ is the unique spread element containing $x$. If we choose $F$ from among the spread elements $G_1,\ldots,G_9$ and pick any 1-dimensional space $x \nsubseteq F$ the cover of $x+F$ is $E$.
	
	\smallskip
	\noindent{\bf Hyperplanes:} Since the only hyperplanes of the $q$-matroid $M_6$ are the spread elements $G_1,\ldots,G_9$, which are pairwise disjoint, axioms (H1) and (H2) hold vacuously. 
	For any 1-dimensional space $x$ and any $i,j$ we have $(G_i \cap G_j)+x = x$, which is contained in some spread element $G_\ell$.

	\smallskip
	\noindent{\bf Open Spaces:} It is easy to see that (O1) and (O2) hold. As regards (O3), consider an open space $O$ of dimension $3$. Let $X\subseteq E$ of codimension $1$. If $O=G_i$ then $O\cap X=D_j$, which is an open space covered by $O$. Otherwise, $O$ does not contain any $D_i$, hence $O\cap X$ also does not. In that case we have that $\{0\}\subseteq O\cap X$ and $O$ covers $\{0\}$.
		
	\smallskip
	\noindent{\bf Spanning Spaces:}
	(S1) and (S2) are easy to verify by looking at Table \ref{Tabellona}.
	Let us look at (S3). We verify it in the case $J=\langle (0,1,0,0,0,0),(0,0,1,0,0,0) \rangle$ and
	$I=J+\langle (0,0,0,1,0,0)\rangle$. It is enough to take, as an example, $X=J+ \langle (0,0,0,0,1,0),(0,0,0,0,0,1),(0,1,1,1,0,0) \rangle$. Of course $J \subseteq X$ and $I\nsubseteq X$. Since $X\cap I =J$ we actually have a spanning space, as required by (S3). 
A very easy example for (S4) is given by taking $A=\langle (0,1,0,0,0,0)\rangle$, $B=\langle (0,0,1,0,0,0) \rangle$, which are both contained in $J=\langle (0,1,0,0,0,0),(0,0,1,0,0,0) \rangle $, their minimal containing spanning space. Their intersection is $\{0\}$ and the required minimal space is $J$ itself.

	\smallskip
	\noindent{\bf Non-spanning Spaces:} (N1), (N2) are easily read from the table. For (N3), we have that the only way for two non-spanning spaces $N_1,N_2$ to have $N_1+N_2\notin\kN$ is if $N_1$ and $N_2$ are in different spread elements. So $N_1\cap N_2=\{0\}$ and $N$ is a $1$-dimensional space, which is a non-spanning space.
\end{example}
	
	We now continue with the characterization of our infinite family of representable $q$-matroids. First we require a well-known lemma 
	\cite[Lemma 3.51]{lidl_niederreiter_1996}.
	
	\begin{Lemma}\label{lem:lidl}
	Let $\alpha_1,...,\alpha_\ell \in \FF_{q^m}$. 
    We have:
    $$ \left|\begin{array}{cccc} 
 	\alpha_1              & \alpha_2   & \cdots   & \alpha_\ell \\ 
 	\alpha_1^q            & \alpha_2^q & \cdots   & \alpha_\ell^{q} \\
 	\vdots             & \cdots  & \vdots   & \vdots      \\
 	\alpha_1^{q^{(\ell-1)}} & \alpha_2^{q^{(\ell-1)}} & \cdots &  \alpha_\ell^{q^{(\ell-1)}}
 \end{array}\right|  
 = \alpha_1 \prod_{j=1}^{\ell-1} \prod_{c_1,\ldots,c_j \in \FF_q} \left( \alpha_{j+1}- \sum_{k=1}^j c_k \alpha_k\right). $$
 In particular, this determinant is nonzero if and only if $\alpha_1,\ldots,\alpha_\ell$
 are linearly independent over $\FF_q$.
	\end{Lemma}
    
    We now describe the $q$-matroid $M[G]=(E,r)$ with rank function defined by $r(A)=\rk(GY)$ for any matrix $Y$ with column space equal to $A$.
    
	\begin{Theorem}\label{th:matex}
		Let $p,s$ be a pair of coprime positive integers and let $m=ps$. Let $E = \FF_{q^m}$ and let
		\[
		G = 
		\left(
		\begin{array}{ccccc}
			1 & \alpha & \alpha^2 & \cdots & \alpha^{m-1} \\
			1 & \alpha^{q^s} & \alpha^{2q^s} & \cdots & \alpha^{(m-1)q^s} \\
			\vdots & \vdots & \vdots & \vdots & \vdots \\
			1 & \alpha^{q^{(p-1)s}} & \alpha^{2q^{(p-1)s}} & \cdots & \alpha^{(m-1)q^{(p-1)s}} \\	
		\end{array}
		\right).
		\]
		Let $e:=(q^m-1)/(q^s-1)$ and for each $i \in \{1,\ldots ,e\}$, let $G_i:=\langle \alpha^i, \alpha^{e+i},\ldots,\alpha^{(s-1)e+i}  \rangle_{\FF_q}$.
		Let $(E,r)$ be the $q$-matroid $M[G]$.
		Let $A$ be a subspace of $\FF_{q^m}$ over $\FF_q$, let $B$ be a basis of $A$ over $\FF_q$, and let 
		$\mathcal{S}:=\{j \in \{1,\ldots,e\} : B \cap G_j \neq \emptyset \}$.
		Let 
		$\mu:=\dim_{\FF_{q^s}}(\langle\alpha^\ell : \ell \in \mathcal{S} \rangle)$.
		Then $r(A) = \min(p,\mu)$.
	\end{Theorem}
	
	\begin{proof}
		
		For each element $\theta \in \FF_{q^m}$, we write $\Gamma(\theta)$ to denote the expression of $\theta$ as a vector of length $m$ in $\FF_q$ with respect to the 
		basis $\{1,\alpha,\ldots,\alpha^{m-1}\}$. We also define $\Gamma(S):=\{\Gamma(s): s \in S\}$ for any $S \subseteq \FF_{q^m}$.
		Let $f(x)=\sum_{k=0}^{m-1} f_i x^i \in \FF_q[x]$. Then $f(\alpha^t) = (1,\alpha^t,\ldots,\alpha^{t(m-1)}) \cdot (f_0,\ldots,f_{m-1})$ for any integer $t$.
		In particular, for the vector $f \in \FF_q^m$, $(Gf)_j = f(\alpha^{q^{(j-1)s}})$ for $1 \leq j \leq p$.
		Now consider the space $G_i = \langle \alpha^i,\alpha^{i+e},\ldots,\alpha^{i+(s-1)e} \rangle \subseteq \FF_{q^m}$. Let $Y$ be the $m \times s$ matrix whose
		$j$-th column is $\Gamma(\alpha^{i+(j-1)e})$. Then, using the fact that $(\alpha^e)^{q^s}=\alpha^e$, we have
		\begin{small}
		\begin{eqnarray*}
			GY &=& 
			\left(
			\begin{array}{cccc}
				\alpha^{i} & \alpha^{i+e} & \cdots & \alpha^{i+(s-1)e} \\
				\alpha^{iq^s} & \alpha^{(i+e)q^s} & \cdots & \alpha^{(i+(s-1)e)q^s} \\
				\vdots & \vdots & \cdots & \vdots \\
				\alpha^{iq^{(p-1)s}} & \alpha^{(i+e)q^{(p-1)s}} & \cdots & \alpha^{(i+(s-1)e)q^{(p-1)s}} 	
			\end{array}    
			\right)
			=
			\left(
			\begin{array}{cccc}
				\alpha^{i} & \alpha^{i+e} & \cdots & \alpha^{i+(s-1)e} \\
				\alpha^{iq^s} & \alpha^{iq^s+e} & \cdots & \alpha^{iq^s+(s-1)e} \\
				\vdots & \vdots & \cdots & \vdots \\
				\alpha^{iq^{(p-1)s}} & \alpha^{iq^{(p-1)s}+e} & \cdots & \alpha^{iq^{(p-1)s} +(s-1)e} 	
			\end{array}    
			\right)\\
			&=&
			\left(
			\begin{array}{cccc}
				\alpha^{i} & \alpha^{i} & \cdots & \alpha^{i} \\
				\alpha^{iq^s} & \alpha^{iq^s} & \cdots & \alpha^{iq^s} \\
				\vdots & \vdots & \cdots & \vdots \\
				\alpha^{iq^{(p-1)s}} & \alpha^{iq^{(p-1)s}} & \cdots & \alpha^{iq^{(p-1)s}} 	
			\end{array}    
			\right)\diag(1,\alpha^e,\alpha^{2e},\ldots,\alpha^{(s-1)e}),
		\end{eqnarray*} 
		\end{small}
		
		\noindent which clearly has rank 1 over $\FF_{q^m}$.
		Let $V$ be an $\Fq$-subspace of $\Fqm$ of dimension $\ell$ and let $B$ be a basis of $V$ over $\Fq$. Each element of $B$ is contained in exactly one spread element $G_i \in \kG$. 
		Write $B =  B_{i_1} \cup \cdots \cup B_{i_t}$, where $B_{i_k} \subset G_{i_k}$ and the $G_{i_k}$ are distinct. In particular, we have $\mathcal{S}=\{i_1,\ldots,i_t\}$.
		Each element of $B_{i_k}$ has the form $\displaystyle \sum_{p=0}^{s-1} f_p \alpha^{i_k + pe} = \alpha^{i_k}f(\alpha^e)$ for some $f(x) \in \FF_q[x]$.
		Let $\displaystyle B_{i_k} = \{ \alpha^{i_k} f^{k,1}(\alpha^e),\ldots,  \alpha^{i_k} f^{k,\ell_k}(\alpha^e) \}$ for some $f^{k,j}(x) \in \FF_q[x]$ where $B_{i_k}$ has order $\ell_k$.
		Let $\displaystyle Y_k = [\Gamma(\alpha^{i_k} f^{k,1}(\alpha^e)),\ldots,  \Gamma(\alpha^{i_k} f^{k,t_k}(\alpha^e))]$, for each $k$. 
		Then 
		\[
		(GY_k)_{j,h} = \alpha^{i_k q^{(j-1)s}} f^{k,h}(\alpha^{eq^{(j-1)s}}) = \alpha^{i_k q^{(j-1)s}} f^{k,h}(\alpha^e).
		\]
		Now let $Y$ be the $m\times \ell$ matrix in $\Fq$ defined by $Y=[Y_1 | \cdots | Y_t]$, so that $GY = [GY_1 | \cdots | GY_t]$.
		We have
		\begin{small}
		\begin{eqnarray*}
			GY&=&
			\left(
			\begin{array}{ccc|c|ccc}
				\alpha^{i_1} f^{1,1}(\alpha^e) & \cdots &  \alpha^{i_1} f^{1,t_1}(\alpha^e) & \cdots  &  \alpha^{i_t} f^{t,1}(\alpha^e) & \cdots &  \alpha^{i_t} f^{1,\ell_1}(\alpha^e)\\
				\alpha^{i_1 q^s} f^{1,1}(\alpha^e) & \cdots &  \alpha^{i_1 q^s} f^{1,t_1}(\alpha^e) & \cdots  &  \alpha^{i_k q^s} f^{t,1}(\alpha^e) & \cdots &  \alpha^{i_k q^s} f^{t,\ell_t}(\alpha^e)\\
				\vdots & \vdots & \vdots & \cdots & \vdots & \vdots & \vdots \\
				\alpha^{i_1 q^{(p-1)s}} f^{1,1}(\alpha^e) & \cdots &  \alpha^{i_1 q^{(p-1)s}} f^{1,t_1}(\alpha^e) & \cdots  &  \alpha^{i_k q^{(p-1)s}} f^{t,1}(\alpha^e) & \cdots &  \alpha^{i_t q^{(p-1)s}} f^{t,\ell_t}(\alpha^e)\\	  
			\end{array}	
			\right) 	
		\end{eqnarray*}
		\end{small}
		Since $f^{k,h}(\alpha^e)\neq 0$ for all $k,h$, $GY$ is column equivalent to the matrix:
		\begin{equation*}
			\left(
			\begin{array}{cccc|cccc|c|cccc}
				\alpha^{i_1} & 0 & \cdots &0 & \alpha^{i_2} & 0 & \cdots & 0 & \cdots & \alpha^{i_t} & 0 & \cdots & 0 \\
				\alpha^{i_1q^s} & 0 & \cdots &0 & \alpha^{i_2q^s} & 0 & \cdots & 0 & \cdots & \alpha^{i_tq^s} & 0 & \cdots & 0 \\
				\vdots & \vdots & \cdots & \vdots & \vdots& \vdots & \cdots & \vdots & \cdots & \vdots & \vdots & \cdots & \vdots \\
				\alpha^{i_1q^{s(p-1)}} & 0 & \cdots &0 & \alpha^{i_2q^{s(q-1)}} & 0 & \cdots & 0 & \cdots & \alpha^{i_tq^{s(q-1)}} & 0 & \cdots & 0 \\
			\end{array}	
			\right).
		\end{equation*}
		Now let $S = \{s_1,\ldots,s_\mu \} \subseteq \mathcal{S}$ such that 
		$\{ \alpha^{s_1},\ldots,\alpha^{s_\mu}\}$ is a basis of $\langle\alpha^\ell : \ell \in \mathcal{S} \rangle_{\FF_{q^s}}$.
		Then by Lemma \ref{lem:lidl}, 
		$$ \left|\begin{array}{cccc} 
 	\alpha^{s_1}              & \alpha^{s_2}   & \cdots   & \alpha^{s_\ell} \\ 
 	\alpha^{s_1q^s}            & \alpha^{s_2q^s} & \cdots   & \alpha^{s_\ell q^s} \\
 	\vdots             & \cdots  & \vdots   & \vdots      \\
 	\alpha^{s_1q^{s(\ell-1)}} & \alpha^{s_2q^{ s(\ell-1)}} & \cdots &  \alpha^{s_\ell q^{s(\ell-1)}}
 \end{array}\right|  
 \neq 0, $$
 where $\ell = \min(p,\mu)$. The result now follows.
	\end{proof}
 
\section{Equivalent Axiom Systems}\label{sec:equivaxiomsyst}

In a number of cases a particular axiom system may have more than one equivalent set of axioms. This is certainly the case for the rank axioms, the hyperplane axioms and the independence axioms. 
Identifying these equivalences can be convenient for various proofs. Alternative axiom systems for the independent spaces and the bases were already given in \cite[Propositions 16 and 40]{JP18}.

\subsection{Independent spaces}

We start with equivalent formulations of the independence axioms: we will show that (I4) can be replaced by either of the following 
alternative axioms.

\begin{itemize}
	\item[(I4')] Let $A \in \kL(E)$ and let $I \in \max(A,\I)$. Let $B \in \kL(E)$. 
	Then there exists $J \in \max(A+B,\I)$ such that in $J \subseteq I+B$.
	\item[(I4'')] Let $A \in \kL(E)$ and let $I \in \max(A,\I)$. 
	Let $x \in \kL(E)$ be a $1$-dimensional space. Then there exists $J \in \max(x+A,\I)$ such that in $J \subseteq x+I$. 
\end{itemize}

These statements are Propositions 14 and 13 of \cite{JP18}, respectively. The proofs in that paper assume the rank axioms, while here we will only use the independence axioms.

\begin{Theorem}\label{lem:I4}
    Let $\I$ be a collection of subspaces satisfying (I1)-(I3).
	Then the axioms systems (I1)-(I4), (I1)-(I4') and (I1)-(I4'') are pairwise equivalent.
\end{Theorem}

\begin{proof}
    Note first that by (I3), if $A \in \kL(E)$ and $I,J \in \max(\kL(A) \cap \I)$,
    then $\dim(I) = \dim(J)$. Therefore, $\max(\kL(U) \cap \I) = \max(U,\I)$ for all $U \in \kL(E)$.
    It is thus clear that (I4) implies (I4'), which implies (I4'').
    Suppose that (I4'') holds. We will show that (I4') holds. Let $A ,B\in \kL(E)$ and let $I\in \max(A,\I)$.
    Suppose that (I4') holds for all subspaces of dimension less than $\dim(B)$.
    Let $C$ be a subspace of $B$ of codimension 1 in $B$ and write $B=x+C$.
    By hypothesis, there exists $J \in \max(A+C,\I)$ such that $J \subseteq I+C$.
    By (I4'') there exists $J' \in \max(A+C+x,\I) = \max(A+B,\I) $ such that $J' \subseteq J + x \subseteq I+C+x=I+B.$
    
	Now suppose that (I4') holds. Let $A,B \in \kL(E)$, let $I \in \max(A,\I)$ and let $J \in \max(B,\I)$. We claim there is member of $\max(A+B,\I)$ that is contained in $I+J$.
	Since $J\in \max(B,\I)$, by (I4') these exists $N \in \max(I+B,\I)$ such that $N \subseteq I+J$.  
	Again by (I4'), there exists $M \in \max(A+B,\I)$ such that $M \subseteq I+B$. But $M\in \max(I+B,\I)$, and hence $M$ and $N$ have the same dimension.
	It follows that $N$ is the required maximal subspace of $A+B$ that is contained in $I+J$ and so (I4') implies (I4). The result follows.
\end{proof} 

We will use (I4'') to establish a cryptomorphism between the independence axioms and the closure axioms in Section \ref{sec:ind-cl}.

The next lemma can be established by repeated applications of (I3). Its proof shows in particular that
if $I,J$ are subspaces of a collection $\I \subseteq \kL(E)$ 
that satisfies the first 3 independence axioms, then if $\dim(J)>\dim(I)$ there exists a subspace $U \subseteq J$ such that $I+U \in \I$, $U \cap I = \{0\}$. This yields an axiom that is equivalent to (I3).

\begin{Lemma}\label{lem:I3}
	Let $\I$ be a collection of spaces satisfying (I1)-(I3). Let $I \subseteq A \in \kL(E)$ and let $I \in \I$.
	Then there exists a subspace $M$ in $\I$ of maximal dimension in $A$, such that $I \subseteq M$.
\end{Lemma}

\begin{proof}
    By (I1), $\max(A,\I)$ is non-empty. 
	Let $J \in \max(A,\I)$. If $\dim(I)=\dim(J)$ then $I$ itself is the required maximal subspace of $A$ in $\I$, so suppose that
	$\dim(I)<\dim(J)$. Then by (I3), there exists $x \subseteq J, x \nsubseteq I$ such that $x+I \in \I$. If $\dim(x+I)=\dim(J)$, then $x+I$ is the required maximal subspace of $A$ in $\I$ that contains $I$. Otherwise, iterative applications of (I3) yields a maximal subspace $M=I+U \in \I$ of $A$,
	with $U \subseteq J$ and $I \cap U = \{0\}$.
\end{proof}

We mention another result that doesn't introduce new equivalent independence axioms, but will arise later in Sections \ref{sec:ind-cl} and \ref{sec:ind-dep}, when we establish cryptomorphisms between the independence and closure axioms and also between the independence and dependence axioms.

\begin{Lemma}\label{lem:M=m+I}
    Let $\I$ be a collection of subspaces satisfying (I1)-(I3).
    Let $A$ be a subspace of $E$ and let $x \in \kL(E),x \nsubseteq A$ be a $1$-dimensional space.
    Let $I \in \max(A,\I)$ and let $M \in \max(x+A,\I)$. Then $\dim(M) \leq \dim(I) +1$.
\end{Lemma}

\begin{proof}
    Clearly, $\dim(I) \leq \dim(M)$. If $\dim(I)=\dim(M)$ then there's nothing to prove, so suppose that $\dim(I) < \dim (M)$. 
    By Lemma \ref{lem:I3}, we may assume that $I \subseteq M$.
    Let $m$ be a $1$-dimensional space such that $m \subseteq M, m \nsubseteq I$. 
    By (I2), $m+I \in \I$.
    By the maximality of $I$ in $A$, we must have $m \nsubseteq A$ 
	and so $A \subsetneq m+A \subseteq x+A$. Therefore, $m+A = x+A$.
	
	We claim that $m+I = M$. 
	Suppose otherwise and let $m'$ be a $1$-dimensional subspace of $M$ that is not contained in $m+I$. Again by (I2), $m+m'+I \in \I$ and by the maximality of $I$ in $A$ we have $m'+m \nsubseteq A$.
	Then $A \subsetneq m'+m+A \subseteq x+A$, so $m'+m+A = x+A=m+A$ and hence
	$m' \subseteq m+A$. 
	Therefore, $m' = \langle \bar{m} +\bar{a} \rangle$ for some $\bar{m}\in m$ and $\bar{a} \in A$.
	If $\bar{m}=0$ then we get the contradiction $m'\subseteq A$. If $\bar{a}=0$ then
	we arrive at the contradiction $m'=m$.
	Therefore, $m+m'+I = m+a+I$ for some $1$-dimensional subspace $a \subseteq A, a \nsubseteq I$, which by (I2) means that 
	$a+I \in \I$, with contradicts $I \in \max(A,\I)$. 
	It follows that $M=m+I$ and that $\dim(M)= \dim(I)+1$. 
\end{proof}

\subsection{Rank function}

The following theorem gives an alternative set of axioms for the rank function. This will be used in Section \ref{sec:rkcl} to show the cryptomorphism between the rank and closure functions. Throughout this section, let $r$ be an integer-valued function defined on the subspaces of $E$. We have the following axioms.
\begin{itemize}
\item[(R1')] $r(\{0\})=0$.
\item[(R2')] $r(A)\leq r(A+x)\leq r(A)+1$.
\item[(R3')] If $r(A)=r(A+x)=r(A+y)$ then $r(A+x+y)=r(A)$.
\end{itemize}
These axioms are sometimes called \emph{local} rank axioms, which explains why we will use a lot of mathematical induction to get to the \emph{global} versions. 

Before proving the equivalence between these axioms and the axioms of the rank function of a $q$-matroid, we state and prove some preliminary results. The first lemma is Proposition 6 from \cite{JP18}. However, the proof in that paper assumes $r$ satisfies (R1), (R2), (R3). Here, we want to use the lemma to prove these axioms, so we re-do the proof of the lemma using (R1'), (R2'), (R3') instead.
\begin{Lemma}
Let $r$ be an integer-valued function defined on the subspaces of $E$ satisfying (R1'), (R2'), (R3'). Let $A,B\in \kL(E)$. If $r(A+x)=r(A)$ for all 1-dimensional spaces $x\subseteq B$, then $r(A)=r(A+B)$.
\end{Lemma}
\begin{proof}
We prove this by induction on $k=\dim B-\dim(A\cap B)$. For $k=0,1,2$ we have $1$-dimensional spaces $x,y\subseteq B$ such that $A+x+y=A+B$ (the sum does not need to be direct). Suppose $r(A+x)=r(A)$ for all 1-dimensional spaces $x\subseteq B$, so in particular, $r(A)=r(A+x)=r(A+y)$. By (R3'), this means $r(A)=r(A+x+y)=r(A+B)$. Now assume the lemma holds for all $A,B$ with $\dim B-\dim(A\cap B)<k$. Suppose $r(A+x)=r(A)$ for all 1-dimensional spaces $x\subseteq B$. Let $B'\subseteq B$ of codimension $2$ and let $x,y\subseteq B$ $1$-dimensional subspaces such that $A+B=A+B'+x+y$. Apply the induction hypothesis to $A$ and $B'$: this gives $r(A+B')=r(A)$. Apply the induction hypothesis also to $A$ and $B'+x$ and to $A$ and $B'+y$, this gives $r(A)=r(A+B'+x)=r(A+B'+y)$. Now we can use (R3') on $x$, $y$ and $A+B'$, giving
\[ r(A)=r(A+B')=r(A+B'+x)=r(A+B'+y)=r(A+B'+x+y)=r(A+B). \]
This proves the induction step, and thus the lemma.
\end{proof}
The next Lemma is the $q$-analogue of Lemma 2.47 of \cite{gordonmcnulty}.

\begin{Lemma}
Let $r$ be an integer-valued function defined on the subspaces of $E$ satisfying (R1'), (R2'), (R3'). Let $A,B\in \kL(E)$. If $A\subseteq B$ then for all $1$-dimensional subspaces $x\in \kL(E)$ we have that $r(A+x)-r(A)\geq r(B+x)-r(B)$.
\end{Lemma}
\begin{proof}
We will prove this statement for $\dim B=\dim A+1$, and the general statement then follows by induction. Let $B=A+y$. From (R2') we know that both sides of the inequality are either $0$ or $1$. If the left hand side is $1$, the inequality is always true, so assume $r(A+x)=r(A)$. We will show that $r(A+y+x)-r(A+y)=0$. By (R2'), $r(A+y)$ is equal to either $r(A)$ or $r(A)+1$. If $r(A)=r(A+y)$, then by (R3') $r(A)=r(A+y+x)$ so in particular, $r(A+y+x)-r(A+y)=0$. If $r(A+y)=r(A)+1$, then we have $r(A)+1=r(A+x)+1=r(A+y)\leq r(A+y+x)$. On the other hand, again by (R2'), $r(A+y+x)$ is either equal to $r(A+x)$ or to $r(A+x)+1$. We conclude that $r(A+y+x)=r(A+x)+1$ and again $r(A+y+x)-r(A+y)=0$.
\end{proof}

We now prove the alternative rank axioms.

\begin{Theorem}\label{t-rank-alternative}
Let $r$ be an integer-valued function defined on the subspaces of $E$. Then $r$ is the rank function of a $q$-matroid $(E,r)$ if and only if $r$ satisfies the axioms (R1'), (R2'), (R3').
\end{Theorem}
\begin{proof}
To prove this theorem, we have to prove that (R1), (R2), (R3) $\Leftrightarrow$ (R1'), (R2'), (R3'). \\
First, assume $r$ satisfies (R1), (R2), (R3). (R1') follows directly from (R1) with $A=\{0\}$. From (R2) it follows that $r(A)\leq r(A+x)$. Lemma 3 from \cite{JP18} in combination with (R1') gives that $r(A+x)\leq r(A)+1$. Together this proves (R2'). (R3') is Proposition 7 from \cite{JP18}. \\
Now we consider the other implication. Assume $r$ satisfies (R1'), (R2'), (R3'). \\
For (R1'), let $A=x_1+x_2+\cdots+x_n$ with $n=\dim A$. Start with (R1') and apply (R2') $n$ times. We do something similar for (R2): let $B=A+x_1+\cdots+x_k$ with $k=\dim B-\dim A$. Now apply (R2') $k$ times. \\
We will prove (R3) by induction on $\dim B-\dim(A\cap B)=k$. Denote $A\cap B=C$. \\
First, let $k=0$, so $C\subseteq A$. Then we have $r(A+C)+r(A\cap C)=r(A)+r(B)$ so (R3) holds. Next, assume that (R3) holds for all $A$ and $B$ with $\dim B-\dim(A\cap B)<k$. Let $B'\subseteq B$ of codimension $1$ such that $A\cap B=A\cap B'$ and let $x$ be a $1$-dimensional subspace such that $B'+x=B$. Then, by the induction hypothesis and the lemma above, we have
\begin{eqnarray*}
r(A+B)+r(A\cap B) & = & r(A+B'+x)+r(A\cap(B'+x)) \\
 & = & r(A+B'+x)+r(A\cap B') \\
 & \leq & r(A+B'+x)-r(A+B')+r(A)+r(B') \\
 & \leq & r(B'+x)-r(B')+r(A)+r(B') \\
 & = & r(A)+r(B).
\end{eqnarray*}
This proves (R3), and thus completes our proof.
\end{proof}

\subsection{Hyperplanes}

We prove a stronger version of the hyperplane axiom (H3). We will use this axiom in Section \ref{FlatHypCryptomorph} when we prove the cryptomorphism between flats and hyperplanes.

Let $\mathcal{H}$ be a collection of subspaces of $E$.
\begin{itemize}
  	 \item[(H3')] For each $H_1,H_2 \in \mathcal{H}$, $H_1\neq H_2$, let $x,y\subseteq E$ be $1$-dimensional spaces with $x \nsubseteq H_1,H_2$
  and $y \subseteq H_1$, $y \nsubseteq H_2$. Then there is an hyperplane $H_3$ such that $(H_1 \cap H_2)+x \subseteq H_3$ and $y \nsubseteq H_3$.
\end{itemize}

\begin{Theorem}\label{t-H3'}
Let $\mathcal{H}$ be a family of subspaces of $E$. 

The family $\mathcal{H}$ satisfies the axioms (H1), (H2), (H3) if and only if it satisfies (H1), (H2), (H3').
\end{Theorem}
\begin{proof}
One direction is clear, since (H3') implies (H3). We will show the 
converse.
Suppose that axioms (H1), (H2), (H3) hold for ${\mathcal{ H}}$. We proceed by induction on the codimension of $H_1\cap H_2$. \\
It cannot be that $\codim(H_1 \cap H_2)=0$ since if so then $H_1=H_2=E$, which violates (H1). Similarly, $\codim(H_1 \cap H_2) \neq 1$, since then one of $H_1,H_2$ must be equal to $E$. The assertion holds void then in these two cases. Suppose now
$\codim(H_1\cap H_2)=d\geq 2$ and that for codimension $d-1$ the assertion holds true. Let us prove it for codimension $d$. \\
Let $H_1,H_2 \in \kH$, $H_1\neq H_2$ and let $x,y$ two spaces of dimension one such that $x \nsubseteq H_1,H_2$ and  $y \subseteq H_2$, $y \nsubseteq H_1$. Using (H3) we can say that there is a hyperplane $H_3 \supseteq (H_1\cap H_2)+x$. If $y \nsubseteq H_3$ we are done, so therefore we suppose $y \subseteq H_3$. This implies $\codim(H_2 \cap H_3)<\codim(H_1\cap H_2)$.
Moreover, since $H_1 \nsubseteq H_3$,
there is a one-dimensional subspace $z\subseteq H_1$ such that $z\nsubseteq H_2,H_3$. Therefore, we can apply the induction hypothesis, finding $H_4 \in \kH$ such that $x \nsubseteq H_4 \supseteq (H_2\cap H_3)+z.$ Since $x \nsubseteq H_1,H_4$, $y \subseteq H_4$, $y \nsubseteq H_1$ and $\codim(H_1\cap H_4)<\codim(H_1\cap H_2)$ we can use (H3') again by the induction hypothesis, and we get a new hyperplane $H_5\in \kH$ such that 
$y \nsubseteq H_5\supseteq (H_1 \cap H_4)+x \supseteq (H_1 \cap H_2)+x,$
from which the result follows.
\end{proof}
\begin{Remark}\label{unione}
In the case of classical matroids, the statement of (H3') is often formulated as follows:
\begin{itemize}
    \item[(H3')] For every $H_1,H_2 \in \kH$ such that $H_1 \neq H_2$ and for every $x \notin H_1 \cup H_2 $, $y \in H_2 \setminus H_1$ there exists $H_3 \in \kH$ such that $y \notin H_3 \supseteq (H_1 \cap H_2) \cup x$.
\end{itemize}
The condition $x \notin H_1 \cup H_2 $ in the classical case is equivalent to saying that $x \notin H_1$ and $x \notin H_2$. However, in the $q$-analogue, saying that $x \nsubseteq H_1+H_2$ is clearly not the same as saying $x \nsubseteq H_1$ and $x \nsubseteq H_2$. We point out that the latter is what we consider in the $q$-analogue.
\end{Remark}

\section{Independent Spaces and the Closure Function}\label{sec:ind-cl}

The goal of this section is to prove that a function satisfying the closure axioms of Definition \ref{closure-axioms} gives rise to a family of independent spaces satisfying the independence axioms of Definition \ref{independence-axioms}. We use this to prove a cryptomorphic description of a $q$-matroid in terms of its closure function, 
As might be expected, the generalisation of the cryptomorphism in the $q$-analogue is non-trivial in this case.

\begin{Lemma}\label{l-closure-incl}
Let $\cl$ be a closure function on $E$ and let $A,B\in \kL(E)$. If $A\subseteq\cl(B)$ then $\cl(A)\subseteq\cl(B)$. In particular, if $B\subseteq A\subseteq\cl(B)$, then $\cl(A)=\cl(B)$.
\end{Lemma}
\begin{proof}
By (Cl2), we have that $\cl(A)\subseteq\cl(\cl(B))$. By (Cl3), $\cl(\cl(B))=\cl(B)$. Combined with applying (Cl2) to $A\subseteq B$, we get equality $\cl(A)=\cl(B)$.
\end{proof}

\begin{Lemma}\label{l-closure-1dim}
Let $x\in \kL(E)$ be a $1$-dimensional space. If $x\subseteq\cl(A)$ then $\cl(A)=\cl(A+x)$.
\end{Lemma}
\begin{proof}
We have $x\subseteq \cl(A)$ and also $A\subseteq\cl(A)$, hence $A+x\subseteq\cl(A)$. This implies $\cl(A+x)\subseteq\cl(A)$. But since $A\subseteq A+x$, we have also that $\cl(A)\subseteq\cl(A+x)$. Hence equality holds.
\end{proof}

We will apply Lemmas \ref{l-closure-incl} and \ref{l-closure-1dim} frequently and not necessarily with direct reference to them.

\begin{Definition}\label{def:closure-indep2}
	Let $\cl$ be a closure function on $E$.
	We say that $I \in \kL(E)$ is an independent space of $(E,\cl)$ if, for each subspace $A\subseteq I$ with $\codim_I(A)=1$, we have $\cl(A)\neq \cl(I)$.
	We write $ \mathcal{I}_{\cl}$ to denote the set of independent spaces of $(E,\cl)$.
\end{Definition}

\begin{Lemma}\label{lem:cl-indep2}
	Let $\cl$ be a closure function on $E$ and
	let $I, J\in \kL(E),I \subseteq J$ satisfy 
	$\dim(J)=\dim(I)+1$. If $I \in \I_{\cl}$ and $J \notin \I_{\cl}$  then $\cl(I)=\cl(J)$.
\end{Lemma}
\begin{proof}
	Since $J \notin \I_{\cl}$, there is a subspace $A \subseteq J$ such that $\codim_J(A)=1$ and $\cl(A)=\cl(J).$ If $I=A$ we are done, so suppose therefore that $A \neq I$ and let 
	$U=A \cap I$. Then $\codim_J(U)=2$ and $\codim_I(U)=\codim_A(U)=1$, so there exist $1$-dimensional spaces $x,y\subseteq J$ such that $I=U+y$ and $A=U+x$. 
	Now, $y \subseteq \cl(J)=\cl(A)=\cl(U+x)$. On the other hand, since $I$ is independent, by definition we have $\cl(U)\subsetneq\cl(I)$. If $y \subseteq \cl(U)$ then $\cl(U)=\cl(U+y)=\cl(I)$, yielding a contradiction. So $y\not\subseteq\cl(U)$.
	By (Cl4) we have
	$$y \subseteq \cl(x+U) \text{ and } y \nsubseteq \cl(U) \implies x \subseteq \cl(U+y)=\cl(I), $$
	which implies that $\cl(I)=\cl(I+x)=\cl(J)$. The result follows.
\end{proof}

\begin{Lemma}\label{lem:xJ}
	Let $ I \in \kL(E)$.
	\begin{enumerate}
		\item $I \in \I_{\cl}$ if and only if every proper subspace $U$ of $I$ satisfies $\cl(U) \subsetneq \cl(I)$.\label{lem:x-1}
			\item Let $A \subseteq I$ such that $\codim_I(A)=1$. If there exists $x \subseteq I$, satisfying $I=x+A$ and $x \subseteq \cl(A)$ then $I \notin \I_{\cl}$.
	\end{enumerate}
\end{Lemma}

\begin{proof}
	Let $U \subsetneq I$. 
	There exists a subspace $W$ of co-dimension one in $I$ such that $U \subseteq W \subseteq I$. 
	Then by (Cl2) we have $\cl(U) \subseteq \cl(W) \subseteq \cl(I)$. If $I \in \I_{\cl}$ then $\cl(U) \subseteq \cl(W) \subsetneq \cl(I)$. 
	Conversely, if every proper subspace of $I$ has closure strictly contained in $\cl(I)$, then in particular this is true of every subspace of co-dimension 1 in $I$, and so $I \in \I_{cl}$ by definition.
	This establishes 1.
	
	Let $x$ be a $1$-dimensional subspace of $I$ such that $x+A=I$.
	If $x \subseteq \cl(A)$ then $\cl(A) = \cl(x+A) = \cl(I)$, and hence $I \notin \I_{\cl}$.
\end{proof}

\begin{Theorem}\label{th:cltoind}
	Let $(E,{\rm cl})$ be a closure function.
	Then $(E,\I_{\cl} )$ satisfies the axioms (I1)-(I4).
\end{Theorem}

\begin{proof}
		Consider first (I1): the space $\{0\}$ does not have any subspaces of codimension one, so the property in the definition of independence holds vacuously. Hence $\{0\}\in\I_{\cl}$ and thus (I1) holds.

		We now show (I2). Let $I\in\I_{\cl}$ and $I'\subseteq I$. We will  show that $I'\in\I_{\cl}$. Let $A'\subseteq I'$ be a subspace of codimension one. Let $A\subseteq I$ be a subspace of codimension one satisfying $A'=A\cap I'$. There is a $1$-dimensional space $x\subseteq I'$, $x \not\subseteq A$ such that $I'=A'+x$ and $I=A+x$. We claim that $\cl(A')\neq\cl(I')$. 
		Suppose not. Then $\cl(I')=\cl(A')\subseteq\cl(A)$ by (Cl2). Since $x\subseteq I'$, it follows that $x\subseteq\cl(A)$. But then $\cl(A)=\cl(A+x)=\cl(I)$, which contradicts the fact that $I\in\I_{\cl}$. Therefore, $\cl(I')\neq\cl(A')$ and $I'\in\I_{\cl}$. 

		Now Let $I,J \in \I_{\cl}$ such that $\dim J > \dim I$. We will show that there exists a $1$-dimensional space $x \subseteq J$, $x\not\subseteq I$ such that $x+I \in \I_{\cl}$.
		This will establish (I3).

	Suppose that (I3) fails for the pair $I,J$. That is, suppose that for any $1$-dimensional $x \subseteq J$, $x\not\subseteq I$, we have $x+I \notin {\I}_{\cl}$.  
	Then from (Cl1) and Lemma \ref{lem:cl-indep2} (which requires (Cl4)), we have that $\cl(x+I) = \cl(I)$ and so in particular, $x \subseteq \cl(I)$ for every $x \subseteq J$, $x \nsubseteq I$. It follows that $J \subseteq \cl(I)$ and hence $\cl(J) \subseteq \cl(I)$, by (Cl2).
	Now suppose further that $\dim(I\cap J)$ is maximal over all such pairs that fail (I3). 
	
	We first note that $I \nsubseteq \cl(U)$ for any proper subspace $U$ of $J$, since otherwise by the independence of $J$ we would have
	$I \subseteq \cl(U) \subsetneq \cl(J)$, which yields the contradiction $\cl(J) \subseteq \cl(I) \subsetneq \cl(J)$. 
	Since $\dim(J) > \dim (I)$, there exists a subspace $A$ of codimension 1 in $J$ such that $I \cap J = I \cap A$.  Since $I \nsubseteq \cl(A),$ there exists some $b \subseteq I$, $b \nsubseteq \cl(A).$ We claim that $b+A \in \I_{\cl}$. As $A \subsetneq J$, by (I2) $A$ is independent. If $b+A$ is not independent, we may apply Lemma \ref{lem:cl-indep2} to deduce that $b \subseteq \cl(b+A) = \cl(A)$, which contradicts our choice of $b \subseteq I, b \nsubseteq \cl(A)$. Write $J'=b+A$. Then as we have just shown, $J'\in \I_{\cl}$ and $\dim(J')=\dim(J)>\dim(I)$. 
	Now $b \nsubseteq \cl(A)$, so in particular, $b \nsubseteq A$ and hence $b \nsubseteq A \cap I = J \cap I$.  
	Moreover, we have
	$ b + (J \cap I) =  b + (A \cap I)  \subseteq (b+A) \cap I = J' \cap I,$
	from which we deduce that $\dim(J' \cap I)>\dim(J \cap I)$.
	By the maximality of $\dim(J \cap I)$ in our hypothesis, it must now be the case that $J'$ and $I$ satisfy (I3). That is, there exists $x \subseteq J'$, $x \nsubset I$ such that $x+I$ is independent.
	Now $x \subseteq J'=b+A$, so $x = \langle \bar{b} + \bar{a} \rangle$ for some $\bar{b} \in b$, and $\bar{a} \in A$.
	Then $x+I = \langle \bar{b} + \bar{a} \rangle+ I = a + I$ for some $a \subseteq A \subseteq J$, since $b \subseteq I$.
	But this contradicts our assumption
	that (I3) fails for $I$ and $J$. We deduce that (I3) holds for $I$ and $J$ and hence holds true in general.
	
	We will establish that (I4'') holds. By Lemma \ref{lem:I4}, this will show that (Cl1)-(Cl4) are sufficient to prove that the axiom (I4) holds for $\I_{\cl}$.
	
	\begin{itemize}
		\item[(I4'')]
		Let $A \in \kL(E)$ and let $I\in \max(A,\I_{\cl})$. Let $x \in \kL(E)$ be a $1$-dimensional space. We will show that $x+A$ has a maximal independent subspace contained in $x+I$.
	\end{itemize}

	If $A=I$ then any subspace of $x+A$ is a subspace of $x+I$, so the result holds.
	Suppose then that $I \subsetneq A$. If $x \subseteq A$ then $\max(x+A,\I_{\cl}) = \max(A,\I_{\cl})$ and so $I$ is the required member of $\max(x+A,\I_{\cl})$ contained $x+I$. Therefore, for the remainder we assume that $x \not{\subseteq} A$.
	Again by the maximality of $I$ in $A$, 
	$a + I \notin \I_{\cl}$ for every $1$-dimensional space $a \subseteq A, a \nsubseteq I$.
	Therefore, by Lemma \ref{lem:cl-indep2}, $\cl(a+I)= \cl(I)$ for every $a \subseteq A$
	and so by (Cl2) and (Cl3), $A \subseteq \cl(I)$.
	
	Let $M \in \max(x+A,\I_{\cl})$. By Lemma \ref{lem:I3}, we may choose $M$ such that $I \subseteq M$.
	If $M=I$ then $I$ itself gives the required subspace of $x+I$
	in $\max(x+A,\I_{\cl})$, so assume that $I \subsetneq M$, i.e. that $\dim(M)>\dim(I)$.
	In particular, this means that $M \nsubseteq A$, by the maximality of $I$ in $A$. 
	
	By Lemma \ref{lem:M=m+I}, $\dim(M)= \dim(I)+1$.
	If $x+I \in \I_{cl}$ then $\dim(x+I)=\dim(M)$, so $x+I\in \max(x+A,\I)$ and (I4) holds.
	We now assume that $x+I \notin \I_{\cl}$.
	
	Since $x+I \notin \I_{\cl}$, by Lemma \ref{lem:cl-indep2}, we have $\cl(x+I) =\cl(I)$.
    Therefore $x \subseteq \cl(I)$ and as we showed above, $A \subseteq \cl(I)$ and so
    $x+A \subseteq \cl(I)$. 
	In particular, $M=m+I \subseteq x+A \subseteq \cl(I)$ and so $\cl(M)=\cl(I)$, which by Lemma
    \ref{lem:cl-indep2} contradicts the independence of $M$.  We deduce that
    $x+I \in \I_{\cl}$ and hence (I4) holds.
\end{proof} 

\begin{Corollary}\label{IndChius}
    Let $(E,\cl)$ be a closure function let $(E,\I)$ be a collection of independent spaces.
    \begin{enumerate}
        \item 

	$(E,\cl)$ determines a $q$-matroid 
	$(E,r)$ whose set of independent spaces is $\I_{\cl}$ and whose closure function satisfies
	$\cl_r = \cl$. 
	
	\item
	Define a function 
	$r_\I: \kL(E) \longrightarrow {\mathbb Z} : A \mapsto \max\{ \dim(I): I \in \I, I \subseteq A \}.$

	Then $(E,\I)$ determines a $q$-matroid $(E,r)$ whose closure function is $\cl_{\I}$ and whose set of independent spaces is $ \I$.
	\end{enumerate}
\end{Corollary}	

\begin{proof}
	We have a closure function $(E,\cl)$,
	which from Theorem \ref{th:cltoind} yields a collection of independent spaces $(E,\I_{\cl})$.  
	From \cite[Theorem 8]{JP18}, $(E,\I_{\cl})$ yields a $q$-matroid $(E,r)$ 
	with rank function defined by
	$r(A):=\max\{\dim I : I\in \I_{\cl} , I \subseteq A \}$ for each $A \in \kL(E)$,
	and whose independent spaces coincide with $\I_{\cl}$.
	Recall that the closure function of 
	$(E,r)$ is defined by
	  $ \cl_{r}(A) :=  \sum_{x \in C_r(A)} x,$
	where $C_r(A)=\{ x \in \kL(E) : \dim(x) =1, r(A+x)=r(A) \}.$
	We claim that $\cl_{r}(A) = \cl(A)$ for each $A \in \kL(E)$.
	Let $A\in \kL(E)$ and let $I \in \max(A,\I_{\cl})$. 
	Then $r(A)=r(I)=\dim(I)$ by definition. Also, $I+a \notin \I_{\cl}$  
	for any $a \subseteq A, a \nsubseteq I$ and so from Lemma \ref{lem:cl-indep2} we have $a \subseteq \cl(a+I) = \cl(I)$. Since $a$ was chosen arbitrarily in $A$, by (Cl2) we get $\cl(A)=\cl(I)$.

	Let $x$ be a $1$-dimensional space such that $x \subseteq \cl(A)$, $x \nsubseteq A$. Clearly $x \subseteq \cl(I)=\cl(A)$, so $x+I \notin \I_{\cl}$ and hence by (I4), 
	$I \in \max(x+A,\I_{\cl})$.
	Then $r(A)=r(I)=\dim I=r(A+x)$ and so $ x \subseteq \cl_{r}(A)$. Therefore $\cl(A) \subseteq \cl_{r}(A)$.
	Now suppose that $x \subseteq \cl_{r}(A), x \not\subseteq A$. Then $r(I)=r(A)=r(A+x)\geq r(I+x) \geq r(I)$ and so $x+I \notin \I.$
	Again by Lemma \ref{lem:cl-indep2} we have 
	$x \subseteq \cl(x+A) = \cl(I) = \cl(A)$ and so $\cl(A)=\cl_{r}(A)$. This proves (1).
	
	By \cite[Theorem 8]{JP18}, the collection of independent spaces $(E,\I)$ determines a $q$-matroid $(E,r)$ whose independent spaces comprise $\I$.
	Define a map $\cl_{\I}:=\cl_{r}$. By \cite[Theorem 68]{JP18}, $\cl_{\I}$ is a closure function, which proves (2).
\end{proof}

\section{Flats and the Closure Function}\label{sec:flat-cl}

We now establish that the flat axioms and the closure axioms are cryptomorphic. 

\begin{Definition}\label{def:flat-cl}
    Let $(E,\kF)$ be a collection of subspaces containing $E$. For each subspace $A \in \kL(E)$, we define
    \[ \cl_\kF(A):=\bigcap \{F:F \in \kF , A \subseteq F \}.\]
\end{Definition}

\begin{Lemma}\label{lem:clisaflat}
    Let $(E,\kF)$ be a collection of subspaces satisfying (F1) and (F2)
    and let $A$ be a subspace of $E$.
    Then $\cl_\kF(A) \in \kF$.
\end{Lemma}

\begin{proof}
    By (F1), $E \in \kF$ and $A \in \kL(E)$, so $\cl_\kF(A)$ is well defined.    
    Observe that $ \cl_\kF(A)$ is a finite-dimensional subspace of $E$.
       Define $\kF_A:=\min(\{F \in \kF : A \subseteq F\})$.
    Clearly, 
    \[ \cl_\kF(A)  = \bigcap \{F: F \in \kF , A \subseteq F \} = \bigcap \{F: F \in \kF_A \}\]
    since every member of $\{F \in \kF : A \subseteq F \}$ contains a member of $\kF_A$.
    Let $F_1,F_2 \in \kF_A$. Then $A, \cl_\kF(A) \subseteq F_1,F_2$ and by (F2), $F_1 \cap F_2 \in \kF$.
    This implies $F_1=F_2$, so $|\kF_A| \leq 1$.
    If $\kF_A$ is empty, then for any subspace $F \in \kF$ that contains $A$ there exists a subspace $F' \in \kF$, $F' \subsetneq F$ and so we could construct the infinite chain of subspaces of $E$, which contradicts the fact that $E$ is finite dimensional.
    Therefore, $\cl_\kF(A) = F$ for the unique $F \in \kF$ satisfying $\kF_A = \{F\}$.
\end{proof}

Before stating and proving the cryptomorphism, we prove a result about closure that will also be used later on in the cryptomorphism between rank and closure.

\begin{Lemma}\label{l-closure-cover}
Let $\cl$ be a closure function on $E$ and let $A,B\in \kL(E)$. If $\cl(A)\subseteq\cl(B)\subseteq\cl(A+x)$ then $\cl(A)=\cl(B)$ or $\cl(B)=\cl(A+x)$.
\end{Lemma}
\begin{proof}
First, note that if $x\subseteq A$, then $\cl(A)=\cl(A+x)$ by Lemma \ref{l-closure-1dim}. This implies $\cl(A)=\cl(B)=\cl(A+x)$ and proves the statement. \\
Assume $x\not\subseteq A$ and suppose, towards a contradiction, that $\cl(A)\subsetneq\cl(B)\subsetneq\cl(A+x)$. Assume also that $x\subseteq\cl(B)$. Since $A\subseteq\cl(A)\subseteq\cl(B)$, we have $A+x\subseteq\cl(B)$. Then Lemma \ref{l-closure-incl} gives that $\cl(A+x)\subseteq\cl(B)\subsetneq\cl(A+x)$, a contradiction. So it needs to be that $x\not\subseteq\cl(B)$. \\
Let $y\in \kL(E)$ be a $1$-dimensional space such that $y\subseteq\cl(B)$, $y\not\subseteq\cl(A)$. Then $y\subseteq\cl(A+x)$ and  axiom (Cl4) gives that $x\subseteq(A+y)$. On the other hand, since both $A\subseteq\cl(B)$ and $y\subseteq\cl(B)$, we have that $A+y\subseteq\cl(B)$ and Lemma \ref{l-closure-incl} gives $\cl(A+y)\subseteq\cl(B)$. This gives a contradiction with $x\subseteq\cl(B)$. In the end, we conclude that it can not happen that $\cl(A)\subsetneq\cl(B)\subsetneq\cl(A+x)$, hence the lemma holds.
\end{proof}

\begin{Theorem}\label{th:clflats}
  Let $(E,\cl)$ be a closure function and  
       let
       $\kF_{\cl}:=\{ F \in \kL(E) : \cl(F)=F \}.$
       Then $(E,\kF_{\cl})$ is a collection of flats.
\end{Theorem}

\begin{proof} 
	We will show that ${\mathcal F}_{\cl}$ satisfies (F1), (F2), and (F3).
	The condition (F1) holds trivially.
	Now let $F_1,F_2 \in {\mathcal F}_{\cl}$. Clearly $F_1 \cap F_2 \subseteq \cl(F_1 \cap F_2)$. Now $F_1 \cap F_2 \subseteq F_1,F_2 $, so by 
	(Cl2) and (Cl3) we have $\cl(F_1 \cap F_2) \subseteq \cl(F_1)=F_1,\cl(F_2)=F_2$. It follows that $F_1 \cap F_2 = \cl(F_1 \cap F_2)$ and hence (F2) holds.
	
	Now let $F\in {\mathcal F}_{\cl}$ and let $x\in \kL(E), x \nsubseteq F$ have dimension $1$. By Lemma \ref{l-closure-cover}, the flat $\cl(F+x)$ is a cover of $F$.

	We claim that $\cl(x+F)$ is unique. 
	Let $F_1$ be a cover of $F$ that contains $x$. Let $y \subseteq F_1$, $y \nsubseteq F$.
	Then $y+F \subseteq F_1$, so by (Cl2) it follows that $\cl(y+F) \subseteq F_1$, and since
	$F \neq \cl(y+F)$ it follows that $F_1 = \cl(y+F)$.
	We now claim that $\cl(x+F) \cap \cl(y+F)=F$. Let $z$ be a $1$-dimensional subspace of $\cl(x+F) \cap \cl(y+F)$ and suppose that $z \not\subseteq F$. Then by (Cl4) we have 
	$z \subseteq \cl(x+F)$ and $z \not\subseteq F$, which implies that $x \subseteq \cl(z+F)$.
	Similarly, $y \subseteq cl(z+F)$.
	But then, applying (Cl2) and (Cl3) , we have 
	   $ \cl(z+F) \subseteq \cl(x+F), \cl(y+F) \subseteq \cl(z+F),$
	and so $\cl(x+F)=\cl(y+F)=\cl(z+F)$.
	We deduce that $\cl(x+F)$ is the required unique flat of $\kF_{\cl}$ that contains $x$ and covers $F$. 
	Therefore (F3) holds. 
	\end{proof}
	
	\begin{Theorem}\label{th:flatscl}
  Let $(E,\kF)$ be a collection of flats
       and let ${\cl_{\kF}}: \kL(E) \to \kL(E)$ be the map defined by
       $\cl_\kF(A):=\bigcap \{F \in \kF : A \subseteq F \}$ for each subspace $A \in \kL(E)$.
       Then $(E,\cl_\kF)$ is a closure function.
\end{Theorem}
	
	\begin{proof}
    We will prove that $\cl_\kF$ satisfies the axioms (Cl1)-(Cl4). For any subspace $A$ of $E$, we define
	$\kF(A):=\{ F \in \kF : A \subseteq F \}$.
	Clearly $A \subseteq  \bigcap \{F \in {\mathcal F} : A \subseteq F \}=\cl_{\kF}(A) $ for any subspace $A \in \kL(E)$, so (Cl1) holds. 
	If $A \subseteq B $ are subspaces of $E$ then any flat $F$ containing $B$ also contains $A$ so $\cl_{\mathcal F}(A)\subseteq \cl_{\mathcal F}(B)$ and thus (Cl2) holds. 
	Let $A$ be a subspace of $E$. We claim that $\cl_{\kF}(\cl_{\kF}(A))=\cl_{\kF}(A)$.
    By Lemma \ref{lem:clisaflat}, $\cl_{\kF}(A)$ is itself a flat. 
	In particular, $\cl_{\kF}(F)=F$ for any $F \in \kF$.
	Therefore $\cl_{\kF}(\cl_{\kF}(A))= \cl_{\kF}(A)$ and so (Cl3) holds.
	
	Let $x,y \in \kL(E)$ be subspaces of dimension 1. Suppose that $y \subseteq \cl_{\mathcal F}(A+x)$, $ y \nsubseteq \cl_{\kF}(A)$. 
	We claim that $x\subseteq \cl_{\mathcal F}(A+y)$. Suppose to the contrary that $x\not\subseteq \cl_{\mathcal F}(A+y)$. 
	Then $\cl_{\kF}(A) \subsetneq \cl_{\kF}(x+A)$ and $\cl_{\kF}(A) \subsetneq \cl(y+A)$.
	By (F3) there is a unique flat $F\in {\mathcal F}$ that covers $A$ and contains $x$. 
	By (Cl2) and (Cl3) we have $A+x \subset \cl_{\mathcal F}(A+x) \subset F$ and so
	$F=\cl_{\mathcal F}(A+x)$. Similarly $\cl_{\mathcal F}(A+y)$ is the unique cover of $A$ that contains $y$. 
	
	Now $y \subset \cl_{\mathcal F}(A+x)$ and $y \subseteq \cl_{\mathcal F}(A)$ by hypothesis and clearly $y \subseteq \cl_{\mathcal F}(A+y)$, so in particular $y$ is contained in two flats that cover $A$. Again by (F3), this means that $\cl_{\mathcal F}(A+x)=\cl_{\mathcal F}(A+y)$, contradicting $x\not\subseteq \cl_{\mathcal F}(A+y)$. We deduce that $x\subseteq \cl_{\mathcal F}(A+y)$. This establishes that (Cl4) holds.
\end{proof}

\begin{Lemma}\label{lem:clflatcomp}
    Let $(E,\cl)$ be a closure function and let $(E,\kF)$ be a collection of flats.
    \begin{enumerate}
        \item For each subspace $A \in \kL(E)$, it holds that 
        $\cl(A) =\bigcap \{B \in \kL(E) : A \subseteq B, B = \cl(B) \}.$ 
        
        \item For each subspace $F \in \kL(E)$, it holds that
        $F \in \kF \Leftrightarrow F = \bigcap \{ K \in \kF : F \subseteq K \}$.
    \end{enumerate} 
\end{Lemma}

\begin{proof}
    Let $A$ be a subspace of $E$. Let $\A:=\{ B \in \kL(E) : A \subseteq B, B = \cl(B) \}$.
    Since $A \subseteq \cl(A)$ by (Cl2), and since $\cl(\cl(A))=\cl(A)$ by (Cl3),  
    we have $\cl(A) \in \A$ and hence $\bigcap \{B : B \in \A\} \subseteq \cl(A)$.
    Conversely, if $B \in \A$, then $A \subseteq B$ and $\cl(B)=B$. Therefore,
    by (Cl2) and (Cl3) we have $\cl(A) \subseteq \cl(B) = B$, so $\cl(A)$ is contained in the
    intersection of all members of $\A$ and we have $\cl(A) = \bigcap \{B : B \in \A\}$.
    This shows that (1) holds.
    
    For each subspace $A \in \kL(E)$, define $\kF(A):=\{K \in \kF : A \subseteq K \}$.
    If $F \in \kF$, then $F \in \kF(F)$ and so $\bigcap \{ K : K \in \kF(F)\}\subseteq F$.
    On the other hand, every member of $\kF(F)$ contains $F$, by definition, so
    $F \subseteq \bigcap \{ K : K \in \kF(F)\}$. Therefore, if $F \in \kF$ then
    $F = \bigcap \{ K : K \in \kF(F)\}$. Conversely, $F = \bigcap \{ K : K \in \kF(F)\}$ is a flat by Lemma \ref{lem:clisaflat}. This shows that (2) holds.
\end{proof}

Given a family of flats $(E,\kF)$ the function $r_{\kF}: \kL(E) \longrightarrow \ZZ$ is defined as follows. For any $A \in \kL(E)$, $r_{\kF}(A)$ is the length minus 1 of the longest chain between $\cl_\kF(\{0\})$ and $\cl_{\kF}(A)$. We recall the following result from \cite[Theorem 3]{WINEpaper1}.

\begin{Theorem}\label{th:bciflatsrk} 
Let $(E,\kF)$ be a family of flats and let $(E,r)$ be a $q$-matroid. Then
$(E,r_\kF)$ is a $q$-matroid whose family of flats is equal to $\kF$.  Conversely,
     \[\kF_r:=\{ A \in \kL(E) : r(A+x)>r(A) \: \forall x\in \kL(E), \dim(x)=1, x \nsubseteq A\},\]
     is a collection of flats for the matroid with rank function $r = r_{\kF_r}$.
\end{Theorem}

\begin{Corollary}\label{cor:cl-flats-rank}
   Let $(E,\cl)$ be a closure function and let $(E,\kF)$ be a collection of flats. 
   Let $\kF_{\cl}$ and $\cl_{\kF}$ be defined in Theorem \ref{th:clflats} and \ref{th:flatscl}.
   \begin{enumerate}
       \item $(E,\cl)$ determines a $q$-matroid  
       with closure function $\cl$ and collection of flats $\kF_{\cl}$.
       \item $(E,\kF)$ determines a $q$-matroid   
       with collection of flats $\kF$ and closure function $\cl_{\kF}$.
   \end{enumerate}
\end{Corollary}

\begin{proof}
     We have a closure function $(E,\cl)$, which
     from Theorem \ref{th:clflats}, yields the collection of flats $(E,\kF_{\cl})$ . By Theorem 
     \ref{cor:cl-flats-rank} $(E,\kF_{\cl})$ determines a $q$-matroid $(E,r)$ with flats $\kF_r=\kF_{\cl}=\{ F \in \kL(E): \cl(F)=F \}$. We claim that $\cl = \cl_r$.
     Let $A$ be a subspace of $E$.
     Recall that
        $\cl_r(A):=\sum_{x \in C_r(A)} x$, where $C_r(A)=\{ x \in \kL(E) : \dim(x)=1, r(A+x) = r(A)\}$.
     It is thus clear that $\cl_r(A) \in \kF_r = \kF_{\cl}$. 
     If $A \subseteq F \in \kF_r$, then $\cl_r(A) \subseteq \cl_r(F) = F$, by the definition of $\kF_{\cl}$.
     Therefore, by Lemma \ref{lem:clflatcomp} (1), we have
     $\cl_r(A) = \bigcap\{F \in \kF_{\cl} : \cl_r(A) \subseteq F \} = \bigcap\{F \in \kF_{\cl} : A \subseteq F \}$.
     On the other hand, from Lemma \ref{lem:clflatcomp} (2) we have
     $\cl(A)=\bigcap\{B \in \kL(E) : \cl(B) = B, A \subseteq B \} = \bigcap\{ B \in \kF_{\cl} : A \subseteq B\},$
     and so $\cl_r(A)=\cl(A)$. This proves (1).
     
     We have a collection of flats $(E,\kF)$, which by Theorem \ref{th:clflats} (2) determines a closure function $(E,\cl_{\kF})$, with $\cl_{\kF}(A):=\{ F \in \kF: A \subseteq F \}$ for $A \in \kL(E)$. By Corollary \ref{IndChius}, $(E,\cl_{\kF})$ determines a $q$-matroid
     $(E,r)$ such that $\cl_{r}=\cl_{\kF}$. Now any $A \in \kF_{r}$ if and only if $A =\cl_{r}(A) = \cl_{\kF}(A) $
     and so $A$ is a flat of $(E,r)$ if and only if $A = \cap \{ F \in \kF : A \subseteq F \}$, which by Lemma \ref{lem:clflatcomp} (2) holds if and only if $F \in \kF$.
     It follows that $(E,\kF)$ and $(E,\cl_{\kF})$ determine the same $q$-matroid with flats
     $\kF$ and closure function $\cl_{\kF}$. This proves (2).
\end{proof}

\section{The Rank and Closure Functions}\label{sec:rkcl}

In this section we prove the cryptomorphism between rank and closure. The main task is to describe the rank function in terms of the closure function.

\begin{Definition}\label{def:rkcl}
Let $\cl$ be a closure function on $E$. Define a function $r_{\cl}:\kL(E)\to\kL(E)$ by \[ r_{\cl}(A)=\min\{\dim(I):\cl(I)=\cl(A),I\subseteq A\}. \]
\end{Definition}

\begin{Definition}
Let $A\in \kL(E)$. A space $I\subseteq A$ such that $\cl(I)=\cl(A)$ and $r_{\cl}(A)=\dim I$ is called a \emph{basis} for $A$.
\end{Definition}

Let us prove some partial results we need in the proof of the cryptomorphism.

\begin{Lemma}\label{l-closure-rank}
Let $\cl$ be a closure function on $E$ and let $A\in \kL(E)$. Then $r_{\cl}(A)=r_{\cl}(\cl(A))$.
\end{Lemma}
\begin{proof}
We have that $r_{\cl}(\cl(A))=\min\{\dim I:\cl(I)=\cl(\cl(A)),I\subseteq\cl(A)\}=\min\{\dim I:\cl(I)=\cl(A),I\subseteq\cl(A)\}$. On the other hand, $r_{\cl}(A)=\min\{\dim I:\cl(I)=\cl(A),I\subseteq A\}$. The set in the definition of $r_{\cl}(A)$ is a subset of the set in the definition of $r_{\cl}(\cl(A))$ and because of (Cl1), the elements of minimal dimension are the same. So $r_{\cl}(A)=r_{\cl}(\cl(A))$.
\end{proof}

\begin{Lemma}\label{l-closure-basis}
Let $\cl$ be a closure function on $E$ and let $A,B\in \kL(E)$. If $A\subseteq B$ and $\cl(A)=\cl(B)$, then $A$ contains a basis for $B$.
\end{Lemma}
\begin{proof}
Let $I$ be a basis for $A$, so $\cl(I)=\cl(A)$. Since also $\cl(A)=\cl(B)$, we have that $\cl(I)=\cl(B)$. Also, by the previous lemma, $r_{\cl}(A)=r_{\cl}(B)=\dim I$. So $I$ is a basis of $B$.
\end{proof}

\begin{Lemma}\label{l-I4closure}
Let $\cl$ be a closure function on $E$ and let $A\in \kL(E)$. If $I$ is a basis for $A$ and $J$ is a basis for $B$, then $A+B$ has a basis contained in $I+J$.
\end{Lemma}
\begin{proof}
This statement follows directly from the proof of Theorem \ref{th:cltoind}. There we use the closure axioms to prove (I4), which is the same statement as this lemma.
\end{proof}

We now have all ingredients for the cryptomorphism between closure and rank.

\begin{Theorem}\label{t-ClosureToRank}
Let $(E,\cl)$ be a closure function. Then $(E,r_{\cl})$ satisfies the axioms (R1)-(R3).
\end{Theorem}
\begin{proof}
Recall that we proved in Theorem \ref{t-rank-alternative} that the  axioms (R1),(R2),(R3) are equivalent to (R1'),(R2'),(R3'). We will prove the latter. \\
It follows straight from the definition that $r_{\cl}(\{0\})=0$, because $\{0\}$ only has subspaces of dimension $0$. This proves (R1'). For (R2') we have to show that $r_{\cl}(A)\leq r_{\cl}(A+x)\leq r_{\cl}(A)+1$. \\
First, suppose there is a basis $J$ of $A+x$ such that $J\subseteq A$. Then $\cl(J)\subseteq\cl(A)$ because of (Cl2) but $\cl(J)=\cl(A+x)$ by definition, so again by (Cl2) we have that $\cl(A)=\cl(A+x)$. This means that also $r_{\cl}(\cl(A))=r_{\cl}(\cl(A+x))$ and because of Lemma \ref{l-closure-rank} we have that $r_{\cl}(A)=r_{\cl}(A+x)$. \\
Next, assume that there is no basis $J$ of $A+x$ such that $J\subseteq A$. Then, without loss of generality, we can write $J=J'+x$ with $J'=J\cap A$. We claim that $J'$ is a basis for $A$. Assuming this claim, we have that $r_{\cl}(A+x)=\dim J=\dim(J')+1=r_{\cl}(A)+1$. Together with the case $J\subseteq A$ we have proven (R2'). \\
We must prove the claim that $J'$ is a basis for $A$. We do this is two steps: first we show that $\cl(J')=\cl(A)$, then we show $\dim J'=r_{\cl}(A)$. Because $J'\subseteq A$, we have that $\cl(J')\subseteq\cl(A)$. Suppose, towards a contradiction, that there is a $1$-dimensional subspace $y$ such that $y\subseteq\cl(A)$ but $y\not\subseteq\cl(J')$. Then $y\subseteq\cl(A+x)=\cl(J'+x)$ but $y\not\subseteq\cl(J')$, so according to (Cl4) $x\subseteq\cl(J'+y)$. Because $J'+y\subseteq\cl(A)$, this means $x\subseteq\cl(A)$. But if both $x$ and $J'$ are in $\cl(A)$, also $J'+x=J\subseteq\cl(A)$, and then we would have the equality $\cl(A)=\cl(A+x)$. However, we assumed there was no basis for $A+x$ contained in $A$. This means we have a contraction, so there is no $y\subseteq\cl(A)$ that is not in $\cl(J')$. Hence $\cl(J')=\cl(A)$. \\
Now all that is left to show is that $\dim J'=r_{\cl}(A)$. Because $\cl(J')=\cl(A)$, we have that $\dim J'\geq r(A)$. Assume, towards a contradiction, that $\dim J'>r_{\cl}(A)$, so $J'$ is not a basis for $A$. According to Lemma \ref{l-closure-basis}, there is an $I\subseteq J'$ that is a basis for $A$. We have that $\cl(I)=\cl(A)$ and $x\not\subseteq\cl(A)$ by assumption, so $\cl(A)\subsetneq\cl(I+x)$. On the other hand, $\cl(I+x)\subseteq\cl(J'+x)=\cl(J)=\cl(A+x)$. Together this gives that $\cl(A)\subsetneq\cl(I+x)\subseteq\cl(A+x)$ and by Lemma \ref{l-closure-cover}, this implies $\cl(I+x)=\cl(A+x)$. But this is a contradiction with $J'+x=J$ being a basis of $A+x$. Hence, $\dim J'=r_{\cl}(A)$ and $J'$ is a basis for $A$, as was required to be shown. \\
Finally, we show that if $r_{\cl}(A)=r_{\cl}(A+x)=r_{\cl}(A+y)$ then $r_{\cl}(A+x+y)=r_{\cl}(A)$. We need only show this for $\dim(A+x+y)=\dim(A)+2$ as the other cases clearly hold. Suppose $r_{\cl}(A)=r_{\cl}(A+x)=r_{\cl}(A+y)$. Then by the proof of (R2'), $A$, $A+x$ and $A+y$ have a common basis $I\subseteq A$. Apply Lemma \ref{l-I4closure} to $A+x$ and $A+y$: this gives that $A+x+y$ has a basis contained in $I+I=I$. Hence $r_{\cl}(A+x+y)\leq r_{\cl}(A+x)=r_{\cl}(A+y)$ and by (R2'), equality holds. This proves (R3').
\end{proof}

\begin{Corollary}\label{cor:clrk}
Let $(\cl,E)$ be a closure function and let $r_{\cl}$ be defined as in Definition \ref{def:rkcl}. Then $(E,r_{\cl})$ is a $q$-matroid with closure function $\cl_{r_{\cl}}=\cl$. Conversely, if $(E,r)$ is a $q$-matroid then $(E,\cl_r)$ is a closure function and $r=r_{\cl_{r}}$.
\end{Corollary}
\begin{proof}
It is proven in \cite[Theorem 68]{JP18} that if $(E,r)$ is a $q$-matroid, then $(E,\cl_r)$ is a closure function. From Theorem \ref{t-ClosureToRank} we know that if $(E,\cl)$ is a closure function, then $(E,r_{\cl})$ is a $q$-matroid with rank function as in Definition \ref{def:rkcl}.
The only thing that remains to be proved is that rank and closure compose correctly, namely $\cl\to r\to\cl'$ implies $\cl=\cl'$ and $r\to\cl\to r'$ implies $r=r'$. The first of these compositions was proven in Corollary \ref{cor:cl-flats-rank}. For the second composition, given a rank function $r$, define $\cl_r(A)$ as in Definition \ref{def-closure}. Then let $r'(A)=r_{\cl_r}(A)=\min\{\dim(I):\cl_r(I)=\cl_r(A),I\subseteq A\}$. Let $J\subseteq A$ be such that $r(A)=\dim(J)=r(J)$. We know that $J$ has minimal dimension with this property. Then $J\subseteq \cl_r(A)$ and $\cl_r(J)=\cl_r(A)$ by Lemma \ref{l-closure-incl}. So by minimality of $J$, $r'(A)=\dim(J)=r(A)$.
\end{proof}

\section{Flats and Hyperplanes}\label{FlatHypCryptomorph}
In this section, we prove the cryptomorphism relating flats and hyperplanes. We start with assuming a collection of flats and deriving a collection of hyperplanes.

\begin{Definition}\label{def-hyperpl-flats}
Let $(E,\kF)$ be a collection of flats. We define a collection of subspaces of $E$ by
       \[ \kH_\kF:=\{H\in \kF : \nexists H'\in \kF \textrm{ such that } H \subsetneq H' \subsetneq E\}. \]
\end{Definition}
Our first aim is to prove that $\kH_\kF$ is a collection of hyperplanes. We first recall the following result from \cite[Section 3]{WINEpaper1}.

\begin{Proposition}\label{LatticeOfFlats}
The members of a collection of flats form a semimodular lattice under inclusion, where for any two flats $F_1$ and $F_2$ the meet is defined to be $F_1\wedge F_2 :=F_1\cap F_2$ and the join $F_1\vee F_2$ is the smallest flat containing $F_1+F_2$. This implies that the lattice of flats satisfies the Jordan-Dedekind property, that is: all maximal chains between two fixed elements of the lattice have the same finite length.
\end{Proposition}

We will show some partial results that we use in the proofs of Theorem \ref{thm:flat-hyper} and Corollary \ref{corr:flat-hyper}.

\begin{Lemma}\label{HyperOneCover}
Let $F\in\kF$. Then $F\notin\kH_\kF$ if and only if $F$ has at least two covers in $\kF$.
\end{Lemma}
\begin{proof}
It follows directly from the definition that if $H\in\kH_\kF$ then is has only the cover $E$, since $E$ is the maximal element in the lattice of flats. For the other direction, suppose that $F\in\kF$ has only one cover. By axiom (F3) there is a unique cover of $F$ for every $x\not\subseteq F$ that contains $x$. If there is only one cover, this cover needs to contain all $x\subseteq E$, $x\not\subseteq F$, as well as $F$ itself. This means the one cover of $F$ is $E$, and hence $F\in\kH_\kF$.
\end{proof}

\begin{Proposition}\label{FlatIsHyperIntersect}
Let $F\in\mathcal{F}$. Then $F$ is the intersection of all $H\in\kH_\kF$ such that $F\subseteq H$.
\end{Proposition}
\begin{proof}
We follow the proof for the classical case as given in \cite[Proposition 2.56]{gordonmcnulty}. For every $F\in\kF$, let $F=F_0\subsetneq F_1\subsetneq\cdots\subsetneq F_n=E$ be a maximal chain between $F$ and $E$. The length of a maximal chain is well defined by Proposition \ref{LatticeOfFlats} and we denote this length it by $n(F)$. (This is in fact the \textbf{corank} or \textbf{nullity} of $F$.) We proceed by induction on $n(F)$. \\
If $n(F)=0$, then $F=E$ and it is the intersection of an empty subset of $\kH_\kF$. If $n=1$ then $F\in\kH_\kF$. Now let $F\in\kF$ with $n(F)=k+1$, $k\geq 1$ and assume that every flat with $n(F)\leq k$ is the intersection of the members of $\kH_\kF$ containing it. Let $I\subseteq\kH_\kF$ be the collection of members of $\kH_\kF$ containing $F$. Clearly $F\subseteq\bigcap_{H\in I}H$. 
We will prove the other inclusion by contradiction. 

Suppose there is a $1$-dimensional $x\subseteq E$ such that $x\subseteq H$ for all $H\in I$ and $x\not\subseteq F$. Because $F$ is not in $\kH_\kF$, it has at least two covers by Lemma \ref{HyperOneCover}. Since $x$ is contained in a unique cover of $F$ by axiom (F3), there is a cover $F'$ of $F$ that does not contain $x$. Let $J\subseteq\kH_\kF$ be the members of $\kH_\kF$ that contain $F'$. Then $J\subseteq I$, because every member of $\kH_\kF$ that contains $F'$ contains $F\subseteq F'$ as well. Clearly, $n(F')=k$, so by the induction hypothesis, $F'=\bigcap_{H\in J}H$. Since $x\not\subseteq F'$, there is an $H\in J$ such that $x\not\subseteq H$. But this is a contradiction of the fact that $x$ is contained in all members of $I$ and $J\subseteq I$. We conclude that $F\subseteq\bigcap_{H\in I}H$ and hence equality holds.
\end{proof}

The above shows that the lattice of flats is co-atomic. We point out the next direct consequence of the proof of  Proposition \ref{FlatIsHyperIntersect}.

\begin{Lemma}\label{FlatInIperp}
Let $F\in\mathcal{F}$ and $F\neq E$. Then there exists an $H\in\mathcal{H}_\mathcal{F}$ containing $F$.
\end{Lemma}

Now we can conclude the first part of our goal, the cryptomorphism from flats to hyperplanes.

\begin{Theorem}\label{thm:flat-hyper}
Let $(E,\kF)$ be a collection of flats and define a collection $\kH_\kF$ as in Definition \ref{def-hyperpl-flats}. Then $(E,\kH_{\kF})$ is a collection of hyperplanes.
\end{Theorem}

\begin{proof}
We will show that $\kH_\kF$ satisfies the axioms (H1), (H2), (H3). Since $E$ cannot be a proper subspace of itself, it is not contained in $\kH_\kF$, which proves (H1). For (H2), let $H_1,H_2\in\mathcal{H}_\mathcal{F}$ satisfy $H_1\subseteq H_2$. Towards a contradiction, assume $H_1\subsetneq H_2$. Since $E\notin\mathcal{H}_\mathcal{F}$ by (H1) and $H_2\in\mathcal{H}_\mathcal{F}$ by assumption, we have that $H_2\neq E$. So for $H_1$, we have $H_1\subsetneq H_2\subsetneq E$ and therefore $H_1\notin\mathcal{H}_\mathcal{F}$. This is a contradiction, so $H_1=H_2$. \\
Now we will prove (H3). Consider two distinct members $H_1,H_2$ of $\mathcal{H}_\mathcal{F}$ and a $1$-dimensional subspace $x\in \kL(E)$. We need to find an $H_3\in\mathcal{H}_\mathcal{F}$ such that $(H_1\cap H_2)+x\subseteq H_3$. By construction of $\mathcal{H}_\mathcal{F}$ we have that $H_1,H_2\in\mathcal{F}$ and so by (F2), $F:=H_1\cap H_2\in\mathcal{F}$. If $x\subseteq F$ then $F+x=F$ and this is contained in some $H_3\in\mathcal{H}_\mathcal{F}$ by Lemma \ref{FlatInIperp}. If $x\not\subseteq F$, then by (F3) there is a unique $F'\in\mathcal{F}$ covering $F$ and containing $x$. Since $F'$ is a flat, again by Lemma \ref{FlatInIperp} it is contained in some $H_3\in\mathcal{H}_\mathcal{F}$. This proves that $H_3\in\mathcal{H}_\mathcal{F}$ satisfies (H3).
\end{proof}

Conversely, we will start with a collection of hyperplanes and show that this collection determines a collection of flats. 

\begin{Definition}\label{def-flats-hyperpl}
Let $(E,\kH)$ be a collection of hyperplanes. Define a collection of subspaces of $E$:
\[ \kF_\kH:=\left\{\bigcap_{H\in I\subseteq\kH} H : I\subseteq \kH \right\} .\]
\end{Definition}

We will prove that $\mathcal{F}_\mathcal{H}$ satisfies axioms (F1)-(F3), having proved some preliminary results. Until stated otherwise, we will assume that $\mathcal{H}$ is a collection of hyperplanes.

\begin{Lemma}\label{lemma-hyp-intersect-1}
Let $I,J\subseteq \kH$. Suppose $I\subseteq J$ and let $F_1:=\bigcap_{H \in I } H$ and $F_2:=\bigcap_{H \in J } H$. Then $F_2\subseteq F_1$.
\end{Lemma}
\begin{proof}
By construction, we have that
\[ F_2=\bigcap_{H \in J} H = \left( \bigcap_{H \in I} H \right) \cap \left( \bigcap_{H \in J\backslash I} H \right) = F_1 \cap \left( \bigcap_{H \in J\backslash I} H \right) \]
and thus $F_2\subseteq F_1$.
\end{proof}

\begin{Lemma}\label{lemma-hyp-intersect-2}
Let $F_1,F_2\in\mathcal{F}_\mathcal{H}$ with $F_2\subseteq F_1$. Let $I\subseteq \kH$ be such that $F_1=\bigcap_{H \in I } H$. Then there is a $J\subseteq \kH$ such that $F_2=\bigcap_{H \in J } H$ and $I\subseteq J$.
\end{Lemma}
\begin{proof}
Let $J$ be the subset of $\kH$ such that $F_2\subseteq H$ for all $H\in J$. Since all the elements of $\mathcal{H}$ containing $F_1$ form a subset of all elements of $\mathcal{H}$ containing $F_2$, and $I$ is a subset of all elements of $\mathcal{H}$ containing $F_1$, we have that $I\subseteq J$.
\end{proof}

\begin{Proposition}\label{prop-cover-hyp-intersect}
Let $F_1,F_2\in\mathcal{F}_\mathcal{H}$ where $F_1$ is a cover of $F_2$ in $\mathcal{F}_\mathcal{H}$. Then there is an $H\in\mathcal{H}$ such that $F_2=F_1\cap H$.
\end{Proposition}
\begin{proof}
Let $I$ be the set of all elements of $\mathcal{H}$ containing $F_1$ so we can write $F_1=\bigcap_{H \in I } H$. By Lemma \ref{lemma-hyp-intersect-2} there is a $J$ such that $F_2=\bigcap_{H \in J } H$ and $I\subseteq J$. Because $I$ contains all the members of $\mathcal{H}$ containing $F_1$ and $F_2\subsetneq F_1$, $I$ is a proper subset of $J$. Let $H$ be a member of  $J\backslash I$ and let $F'=F_1\cap H$. Then $F'\subsetneq F_1$ because $H$ does not contain $F_1$. Write $I'=I\cup\{H\}$, so $F'=\bigcap_{H \in I' } H$. By Lemma \ref{lemma-hyp-intersect-1} we have that $F_2\subseteq F'$. Combining gives that $F_2\subseteq F'\subsetneq F_1$. But $F_2\subseteq F_1$ is a cover, so $F_2=F'$ and thus $F_2=F_1\cap H$.
\end{proof}

Note that the converse of this statement is not true: if $F_1\cap H=F_2$, then $F_1$ need not cover $F_2$. The following example illustrates this.

\begin{example}
Let us consider $E=(\FF_2)^5$ and denote by $e_i$, $1 \leq i \leq 5$ the element in the canonical basis of $E$ with $1$ in position $i$ and zeroes in all the other positions.
Consider the uniform $q$-matroid $U_{4,5}(\mathbb{F}_2)$ of rank $4$ on $E$ (see \cite[Example 4]{JP18}).
Clearly all $3$-subspaces are hyperplanes.
If we consider $F_1=\langle e_1,e_2,e_3\rangle$ this is then a hyperplane and so also a flat.
Take then $H_1=\langle e_2,e_3,e_4\rangle$
and  $H_2=\langle e_2,e_4,e_5\rangle$. Let
$F':=F_1 \cap H_1$ and $F_2:=F_1 \cap H_2$.
Then we have $F_2 \subsetneq F' \subsetneq F_1$ and the number of hyperplanes over $F_2$ is one more the number of hyperplanes over $F_1$. Therefore $F_1 \neq F_2$ and $F_1$ is not a cover of $F_2$.
\end{example}

In the next lemma we use the hyperplane axiom (H3'). Recall that in Theorem \ref{t-H3'} it is proven that the axioms (H1), (H2), (H3') are an equivalent set of axioms for a collection of hyperplanes. 

\begin{Lemma}\label{IntersIsF}
Let $F\in\mathcal{F}_\mathcal{H}$ and let $J\subseteq \kH$ be the set of elements of $\mathcal{H}$ containing $F$. Let $x$ be a $1$-dimensional space not contained in $F$.
Let $I \subseteq J$ be the set of all elements of $J$ containing $x$. For each $H' \in J \setminus I$, let $F'=(\bigcap_{H \in I} H)\cap H'$. Then $F'=F$.
\end{Lemma}
\begin{proof}
Suppose by contradiction that there is  $H' \in J \setminus I$,  with $F' \neq F$. In particular $F' \supsetneq F$ and there should be a $y$, $\dim(y)=1$, such that $y \nsubseteq F$ but $y \subseteq F'$. Since $F'=(\bigcap_{H \in I} H)\cap H'$, then $y \subseteq H$ for each $H \in I$ and $y \subseteq H'$.
For $H'$, then, we know that does not contain $x$ but it contains $y$. All the $H$ contain $y$ too. But $y$ is not in $F$ so there is some  $H_u \in J \setminus I$ such that $y \nsubseteq H_u$. Consider $H$ and $H_u$.
We know that $x \nsubseteq H',H_u$ and $y \subseteq H'$, $y \nsubseteq H_u$. By (H3') there is a hyperplane $\overline{H}$ containing  $(H_u \cap H') +x$, but $y \nsubseteq \overline{H}$. Therefore $\overline{H}$ contains $F$ and $x$ so it should be an element of $I$, contradicting that all elements of $I$ contain $y$.
\end{proof}

After all this ground work, we can now prove the cryptomorphism from hyperplanes to flats.

\begin{Theorem}
Let $(E,\kH)$ be a collection of hyperplanes and define a collection $\kF_\kH$ as in Definition \ref{def-flats-hyperpl}. Then $(E,\kF_{\kH})$ is a collection of flats.
\end{Theorem}
\begin{proof}
We will prove that $\kF_{\kH}$ satisfies the flat axioms (F1),(F2),(F3). $E$ is a flat since it is the empty intersection 
of hyperplanes, hence (F1) holds. Let $F_1:=\bigcap_{H \in I} H$ and  $F_2:=\bigcap_{j \in J} H$ two elements in $\kF_\kH$. Then $F_1 \cap F_2=\bigcap_{H\in I\cup J} H$ and so $F_1 \cap F_2 \in \kF$. This proves (F2). \\
Now we come to (F3). We have a flat $F\in\mathcal{F}_\mathcal{H}$ and $x \nsubseteq F$. We want to prove the existence of a unique $F'$ which contains $x$ and covers $F$.
We take $J$, the set of all hyperplanes containing $F$ and we consider the intersection $F'$ of all the hyperplanes in $J$ which also contain $x$.
Now we can use the Lemma \ref{IntersIsF} to see that such a flat is a cover: being $x \nsubseteq F$ there is also an element of $J$ not containing $x$ but for all of them the intersection is $F$.\\
The uniqueness of the flat covering $F$ and containing  $x$ can be easily proved by contradiction.\\
Suppose there is another flat $F''$ covering $F$ and containing  $x$. Then, by (F2), $F''':=F'\cap F''$ is a flat, which obviously contains $F$ and $x$ and it is contained in $F'$ and $F''$. This contradicts the fact that $F'$ covers $F$. We therefore conclude that $F'$ is unique. This completes the proof of (F3).
\end{proof}

\begin{Corollary}\label{corr:flat-hyper}
Let $(E,\kF)$ be a collection of flats and let $(E,\kH)$ be a collection of hyperplanes.
\begin{enumerate}
\item $(E,\kF)$ determines a $q$-matroid with collection of flats $\kF$ and collection of hyperplanes $\kH_\kF$.
\item $(E,\kH)$ determines a $q$-matroid with collection of hyperplanes $\kH$ and collection of flats $\kF_\kH$.
\end{enumerate}

\end{Corollary}
\begin{proof}
To prove (1), we show that $\kF_{\kH_\kF}=\kF$. Let $F\in\kF$. We have to prove that $F\in\kF_{\kH_\kF}$. From Proposition \ref{FlatIsHyperIntersect} we know that $F$ is the intersection of all members of $\kH_\kF$ containing $F$. This intersection is in $\kF_{\kH_\kF}$ by definition, so $F\in\kF_{\kH_\kF}$.

For the other inclusion, start with $F\in\kF_{\kH_\kF}$. We can find a finite chain
$F\subsetneq F_1\subsetneq F_2\subsetneq\cdots\subsetneq F_k\subsetneq E $
of flats of $\kF_{\kH_\kF}$ by using axiom (F3) multiple times. By Proposition \ref{prop-cover-hyp-intersect} we can find a hyperplane $H_1\in\kH_\kF$ such that $F=F_1\cap H_1$. In the same way we find $H_i\in\kH_\kF$ such that $F_{i-1}=F_i\cap H_i$ for all $1\leq i\leq k$. This gives that $F=\bigcap_{i\in\{1,\ldots,k\}}H_i$, a finite intersection. By applying the axiom (F2) multiple (but a finite number of) times, we get that $F\in\kF$. This shows $\kF_{\kH_\kF}\subseteq\kF$ and hence equality holds. \\
To show $\kH_{\kF_\kH}=\kH$ for part (2), let $H\in\kH$ and let $I=\{H\}\subseteq\kH$. Then $H=\bigcap_{H'\in I}H'$ hence $H\in\kF_\kH$. We need to show that $H\in\kH_{\kF_\kH}$. This is the case if there is no flat in $\kF_\kH$ that covers $H$ and is not equal to $E$. This is impossible, since $H$ is the intersection of only one hyperplane. Therefore, $\kH \subseteq \kH_{\kF_\kH}$. 
Now suppose that $H \in \kH_{\kF_\kH}$. Then $H$ is a maximal element of $\kF_\kH = \{ \bigcap_{H \in I} : I \subseteq \kH \}$. But then in particular, $H \in \kH$. 
It follows that $\kH_{\kF_\kH}=\kH$. \\
We know from Theorem \ref{th:bciflatsrk} that $(E,\kF)$ determines a $q$-matroid with $r=r_\kF$ and $\kF=\kF_{r_\kF}$.
It follows from the above that also $(E,\kH)$ defines a $q$-matroid with $\kH=\kH_{\kF_r}$. It follows directly from Definition \ref{def:flat} that $\kH_{\kF_r}=\kH_r$. Hence we have a $q$-matroid in both parts (1) and (2).
\end{proof}

\section{Dependence and Independence}\label{sec:ind-dep}

We now establish that the independence axioms and the dependence axioms are cryptomorphic.
It is worth noting at this point that while we require the four axioms (I1)-(I4) in order to define a $q$-matroid, the three dependence axioms (D1)-(D3) are sufficient.

\begin{Theorem}\label{th:ind2dep}
	Let $(E,\I)$ be a collection of independent spaces. Let $\kD=\opp(\I)$. Then $\kD$ is a collection of dependent spaces.  
\end{Theorem}
\begin{proof}
	By (I1), $\I$ is non-empty. Let $I \in \I$. Then $0 \in \I$ by (I2), so $0 \notin \kD$ and (D1) holds.
	Let $D_1 \in \kD$ and let $D_1 \subseteq D_2 \in \kL(E)$. 
	Then $D_2 \notin \I$ by (I2), so (D2) holds. 
	
	Now let $D_1,D_2 \in \kD$ such that $D_1 \cap D_2 \in \I$.
	By Lemma \ref{lem:I3}, there exist $I_1 \in \max(D_1,\I)$ and $I_2\in \max(D_2,\I)$
	such that $D_1 \cap D_2 \subseteq I_1,I_2$ and clearly $I_i \subseteq D_i$ for $i =1,2$.  
	Then we have $D_1 \cap D_2 = I_1 \cap I_2$.
	Moreover, $\dim(I_1)\leq \dim(D_1)-1$ and $\dim(I_2)\leq \dim(D_2)-1$ since $I_i \notin \kD$ for $i=1,2$.
	
	Let $D\subseteq D_1+D_2$ have codimension one in $D_1+D_2$.
	Suppose now, towards a contradiction, that $D\in \I$.
	By (I4), $D_1+D_2$ has a maximal independent subspace $V$ contained in $I_1+I_2$, which, by maximality, satisfies
	$\dim(V) \geq \dim(D) = \dim(D_1+D_2)-1$. Therefore,
	$\dim(D_1+D_2)-1 \leq \dim(V) \leq \dim(D_1+D_2)-1,$
	and so $\dim(D_1+D_2)-1 =\dim(V) \leq\dim(I_1+I_2)$. 
	We have
	\begin{align*}\label{eq:int}
	\dim(D_1 + D_2)-1 & \leq \dim(I_1+I_2) ,\\
	                  & = \dim(I_1)+\dim(I_2) - \dim(I_1 \cap I_2),\\
	                  &\leq  \dim(D_1)+\dim(D_2)-\dim(D_1 \cap D_2) -2, \\
		              &= 	\dim(D_1+D_2) -2, 
	\end{align*}  
    yielding the required contradiction. It follows that (D3) holds.
    \end{proof}

\begin{Theorem}\label{th:dep2ind}
	Let $(E,\kD)$ be a collection of dependent spaces. Let $\I=\opp(\kD)$. Then $\I$ is a family of independent spaces.
\end{Theorem}

\begin{proof}
	Since $\{0\} \notin \kD$, $\{0\} \in \I$ and so (I1) holds. If $J \in \I$ and $I \subseteq J$ then from (D2) it must be the case that $I \in \I$, so (I2) holds.
	
	Now let $I,J \in \I$ such that $\dim(I)<\dim(J)$. 
	We will apply induction on $\dim(I / (I \cap J))$. If $\dim(I / (I \cap J))=0$ then $I \subseteq J$ and so clearly (I3) holds for the pair $I,J$.
	Now let $k$ be a non-negative integer and suppose that (I3) holds for all subspaces $U,V \in \I$ satisfying $\dim(V)>\dim(U)$ and $\dim(U/(U \cap V)) \leq k$.
	Suppose that $\dim(I / (I \cap J))=k+1$.
	We claim that there exists $x \subseteq J, x \nsubseteq I$ such that $x+I \in \I$.  
	
	Let $I_1$ be a subspace of codimension one in $I$ that contains $I \cap J$.
	Then $I_1 \cap J=I \cap J$, $\dim(J)>\dim(I_1)$ and by (I2),  $I_1 \in \I$. Since
	\[
	\dim(I_1/(I_1 \cap J))=\dim(I_1)-\dim(I_1 \cap J)=\dim(I)-1 -\dim(I \cap J) =k,
	\]
	by the induction hypothesis (I3) holds for the pair $I_1,J$; that is, there exists a $1$-dimensional space $a \subseteq J, a \nsubseteq I_1$ such that 
	$I_2=a+I_1 \in \I$. 
	We have $a \subseteq J, a  \nsubseteq I_1$ and so $ a \nsubseteq I_1 \cap J = I \cap J$. 
	Clearly, $\dim(I_2) = \dim(I)$.
	We have 
	$I_2 \cap J = (a+I_1) \cap J = a + (I_1 \cap J) = a + (I \cap J),$
	and $a \nsubseteq I \cap J$, so $\dim(I_2 \cap J) = \dim(I \cap J)+1$.
	Therefore, 
	\[
	\dim(I_2/(I_2 \cap J)) = \dim(I_2)-\dim(I_2 \cap J) = \dim(I) - \dim(I \cap J) -1 = k,
	\] 
	hence again by hypothesis, there exists a $1$-dimensional space
	$b \subseteq J$, $b \nsubseteq I_2$ such that $I_3 = b+I_2 \in \I$.
	Clearly, $a \nsubseteq I$.	
	Also, $b \subseteq J, b  \nsubseteq I_2$ and  
	since $a+ (I \cap J) = I_2 \cap J$, it follows that $b \nsubseteq a+I$. 
	Therefore, $a+I \neq b+I$ and so
	we have $(a+I) \cap (b+I) = I \in \I$.
	If both $a+I, b+I \in \kD$, then by (D3) every subspace of codimension 1 in $a+b+I$ is dependent. 
    Now $\dim(a+b+I) = \dim(I)+2$ and $\dim(I_3) = \dim(I_1)+2$, so $\codim_{a+b+I}(I_3)=1$, so in particular
	this implies that $I_3 \in \kD$, which contradicts $I_3 \in \I$.
	We deduce that at least one of $a+I$ or $b+I$ is in $\I$. This establishes (I3).
	
	Let $A$ be a subspace of $E$ and let $x \in \kL(E)$ be a $1$-dimensional space.
	Let $I \in \max(A,\I)$.
	We claim that there exists a member of $\max(x+A,\I)$ contained in $x+I$. 
	This will prove that (I4) holds, by Lemma \ref{lem:I4}.
	If $A=I$ then any subspace of $x+A$ is a subspace of $x+I$, so the result holds.
	Suppose then that $I \subsetneq A$. If $x \subseteq A$ then $\max(x+A,\I) = \max(A,\I)$ and so $I$ is the required member of $\max(x+A,\I)$ contained $x+I$. Therefore, for the remainder we assume that $x \not{\subseteq} A$.
	
	Let $M \in \max(x+A,\I)$. By Lemma \ref{lem:I3}, we may choose $M$ such that $I \subseteq M$.
	If $M=I$ then $I$ gives the required subspace
	in $\max(x+A,\I)$, so assume that $I \subsetneq M$, i.e. that $\dim(M)>\dim(I)$.
	In particular, this means that $M \nsubseteq A$, by the maximality of $I$ in $A$. Furthermore, 
	$y + I \notin \I$ for any $1$-dimensional space $y \subseteq A, y \nsubseteq I$.

	By Lemma \ref{lem:M=m+I}, $\dim(M)= \dim(I)+1$ and so $M=m+I$ for some $1$-dimensional space
	$m \nsubseteq A$.
	
	If $x \subseteq M$, then as $x \nsubseteq I$ and since $\codim_M(I)=1$, we have $M=m+I=x+I$ and $M$ gives the subspace
	satisfying (I4), so suppose that $x \nsubseteq M$.
	If $x+I \in \I$ then $x+I \in \max(x+A,\I)$, so suppose otherwise, i.e. that $x+I \in \kD$.
	We have $m+I \in \max(x+A,\I)$ and $m \subseteq x+A$, $m \nsubseteq A$.
	Therefore, $m= \langle \bar{x} +\bar{a} \rangle$ for $x=\langle \bar{x} \rangle$ and a $1$-dimensional subspace $a=\langle \bar{a} \rangle \subseteq A$.
	By the maximality of $I$ in $A$, $a+I \in \kD$. Now $(x+I) \cap (a+I) = I \in \I$, since $x \nsubseteq A$ and $I$ has codimension
	1 in both $a+I$ and $x+I$. Then by (D3), every subspace of codimension 1 in $a+x+I$ is a member of $\kD$. 
	But as $m+I  \subseteq x+a+I$ is independent, by assumption, and as it has codimension 1 in $x+a+I$, we arrive at a contradiction.
	We deduce that $\dim(M) = \dim(I)$ and so (I4) holds.
\end{proof}

\begin{Corollary}\label{cor:ind-dep}
   Let $(E,\I)$ be a collection of independent spaces and let $(E,\kD)$ be a collection of
   dependent spaces. 
   Suppose that $\kD = \opp(\I)$.
   Then
   $(E,\I)$ and $(E,\kD)$ each determine the same $q$-matroid $(E,r)$ such that
   $\kD$ is the collection $\kD_r$ of dependent spaces of $(E,r)$ and $\I$ is the collection of
   independent spaces $\I_r$ of $(E,r)$.
\end{Corollary}

\begin{proof}
   By \cite[Theorem 8]{JP18}, $(E,\I)$ determines a $q$-matroid $(E,r)$ such that 
   $\I=\I_r$. Since $\kD=\opp(\I)$, we have
   $D \in \kD$ if and only if $D \notin \I_r$ and in particular $\kD$ must be the set of dependent spaces
   of $(E,r)$.
\end{proof}

\section{Dependent Spaces and Circuits}\label{sec:dep-circ}

Recall that for a collection of subspaces $\kS$ of $E$ that $\min(\kS) = \{ A \in \kS : B \nsubseteq A, \text{ any }B \in \kS, A \neq B \}$.

\begin{Lemma}\label{DC1}
	Let $(E,\kD)$ be a collection of dependent spaces. Let $\kC=\min(\kD)$. Then $(E,\kC)$ is a collection of circuits of $E$.
\end{Lemma}
	
\begin{proof}
    Since $\kC=\min(\kD)$, we have $\kC \subseteq \kD$. Therefore, $\{0\} \notin \kC$ by (D1), which gives (C1). Let $C_1,C_2 \in \kC$ such that $C_1 \subseteq C_2$.
    By the definition of $\min(\kD)$, $C_1$ is not properly contained in any other member of $\kC$. It follows that $C_1=C_2$ so that (C2) holds.
    
    Now let $C_1,C_2 \in \kC$ with $C_1 \neq C_2$ we claim that every space of codimension 1 in
    $C_1+C_2$ contains a circuit.
    Since $C_1 \neq C_2$, we have $C_1 \cap C_2 \subsetneq C_1,C_2$, therefore, by (C2) we have $C_1 \cap C_2 \notin \kC$ and in particular is not a dependent space.
    Then by (D3), there is a space $D \in \kD$ of codimension one in $C_1+C_2$.
    Let $C_3 \in \kD$ be a subspace of $D$ such that no member of $\kD$ is contained in $C_3$. 
    Such a space clearly exists since $E$ has finite dimension. Then $C_3 \in \min(\kD) = \kC$
    and so (C3) holds.     
\end{proof}	

\begin{Lemma}\label{DC2}
    Let $(E,\kC)$ be a collection of circuits. Let $\kD=\upp(\kC)$. Then $(E,\kD)$ is a collection of dependent subspaces of $E$.	
\end{Lemma}

\begin{proof}
	By (C1), $\{0\} \notin \kC$ and so in particular $\{0\} \notin \kD$ and so (D1) holds. If $D_1 \subseteq D_2$ and $D_1 \in \kD$ then
	there exists $C \in \kC$ contained in $D_1$, by the definition of $\upp(\kD)$, and so $C \subseteq D_2$ which gives $D_2 \in \kD$.
	This shows that (D2) holds.
	Now let $D_1,D_2 \in \kD$ such that $D_1 \cap D_2 \notin \kD$.
	Let $H$ be a subspace of codimension 1 in $E$.
	We claim that $(D_1+D_2) \cap H \in \kD$.

	There is no circuit contained in $D_1 \cap D_2$ by definition of $\upp(\kC)$.
	Let $C_1$ and $C_2$ be circuits contained in $D_1$ and $D_2$, respectively. We have $C_1 \neq C_2$, since otherwise
	$D_1 \cap D_2$ contains a circuit.

	By (C3), there exitst a circuit $ C_3 \subseteq (C_1+C_2) \cap H$.
	Then clearly $(D_1+D_2) \cap H \in \kD$, since 
	$C_3 \subseteq (C_1+C_2)  \cap H \subseteq(D_1+D_2) \cap H$ implies $(D_1+D_2) \cap H \in \upp(\kC)=\kD$.
\end{proof}

\begin{Corollary}\label{DC}
   Let $(E,\kD)$ be a collection of dependent spaces and let $(E,\kC)$ be a collection of
   circuits such that $\kD=\upp(\kC)$.
   Then $(E,\kD)$ and $(E,\kC)$ both each determine a $q$-matroid $(E,r)$ whose
   collection of dependent spaces is $\kD$ and whose collection of circuits is $\kC$.
\end{Corollary}

\begin{proof}
   By Corollary \ref{cor:ind-dep}, $(E,\kD)$ determines a $q$-matroid whose dependent spaces comprise $\kD$. The result now follows since $\kC = \min(\kD)$.
\end{proof}

\section{Hyperplanes and (Co)circuits}\label{sec-HC}

We will prove a cryptomorphism between cocircuits and hyperplanes, implying a cryptomorphism between hyperplanes and circuits. We call $\kC^*$ the family of \textbf{cocircuits} of a $q$-matroid.

\begin{Theorem}
Let $\kC^*$ and $\kH$ be two families of subspaces of $E$ such that $\kC^*=\kH^\perp$. Then $(E,\kH)$ is a collection of hyperplanes if and only if $(E,\kC^*)$ is a collection of circuits.
\end{Theorem}
\begin{proof}
Suppose $\kH$ is a collection of hyperplanes, so it satisfies the hyperplane axioms. Since $\kC^*=\kH^\perp$, we get that $\kC^*$ satisfies the circuit axioms by taking orthogonal complements in all the hyperplane axioms. Since $(\kH^\perp)^\perp=\kH$, we get the other implication by taking orthogonal complements again.
\end{proof}

\begin{Remark}
Recall that in Theorem \ref{t-H3'} we proved that the axioms (H1), (H2), (H3) are equivalent to the axioms (H1), (H2), (H3'). From the theorem above it follows that the axioms (C1), (C2), (C3) are equivalent to the axioms (C1), (C2), (C3'), with (C3') equal to the following:
 \begin{itemize}
     \item[(C3')] For distinct $C_1,C_2 \in \kC$ and any $X,Y\in \kL(E)$ of codimension $1$ with $X\nsupseteq C_1,C_2$, $Y \supseteq C_1$, $Y\nsupseteq C_2$, there is a circuit $C_3 \subseteq \kC$ such that $C_3 \subseteq (C_1+C_2)\cap X$ and $Y\nsupseteq C_3$.
 \end{itemize}
\end{Remark}

\begin{Corollary}\label{HypCircq-Matr}
Let $(E,\kH)$ be a collection of hyperplanes and let $(E,\kC^*)$ be a collection of circuits such that $\kC^*=\kH^\perp$. Then $(E,\mathcal{H})$ and $(E,\mathcal{C}^*)$ both each determine a $q$-matroid $(E,r)$ whose collection of hyperplanes is $\mathcal{H}$ and whose collection of cocircuits is $\mathcal{C}^*$.
\end{Corollary}
\begin{proof}
By Corollary \ref{corr:flat-hyper}, $(E,\kH)$ determines a $q$-matroid whose hyperplanes comprise $\mathcal{H}$. The result now follows since $\kC^* = \kH^\perp$.
\end{proof}

As the result above suggest, cocircuits are closely related to circuits. This is made precise by the results below. First we prove a small lemma.

\begin{Lemma}\label{l-hyper-basis}
A hyperplane is a maximal space with respect to not containing a basis.
\end{Lemma}
\begin{proof}
A hyperplane $H$ has rank $r(M)-1$ and is rank-maximal because it is a flat. This means that for all $1$-dimensional spaces $x\not\subseteq H$ we have that $r(H+x)=r(H)+1=r(M)$ and thus $H+x$ contains a basis.
\end{proof}

\begin{Proposition}\label{CircuitsAreCocirc}
The circuits of the matroid $M$ are the cocircuits of the dual matroid $M^*$.
\end{Proposition}
\begin{proof}
We follow Proposition 3.18 of \cite{gordonmcnulty}. We use that for subspaces if $A\subseteq B$ then $B^\perp\subseteq A^\perp$. The following are equivalent (see Theorem \ref{thm-dualbases} and Lemma \ref{l-hyper-basis}): \\[6pt]
\begin{tabular}{l @{\ \ $\Leftrightarrow$\ \ } l}
$C$ is a circuit of $M$ & $C$ is a minimal dependent space in $M$ \\
 & $C$ is minimal with respect to not being contained in any basis $B$ of $M$ \\
 & $C^\perp$ is maximal with respect to not containing any $B^\perp$ \\
 & $C^\perp$ is maximal with respect to not containing a basis $B^\perp$ of $M^*$ \\
 & $C^\perp$ is a hyperplane of $M^*$ \\
 & $C$ is a cocircuit of $M^*$\qedhere
\end{tabular}
\end{proof}

From this proposition it follows directly that the circuits of a $q$-matroid are a collection of circuits.

\begin{Corollary}
Let $(E,\kC^*)$ be the collection of cocircuits of a $q$-matroid $M$. Then $(E,\kC)$ is the collection of circuits of $M^*$.
\end{Corollary}

\begin{Remark}
In \cite[Theorem 64]{JP18} the following statement, which is a variation on (C3), is proven for a $q$-matroid:
\begin{itemize}
\item[$\overline{\mbox{(C3)}}$]
For distinct $C_1,C_2 \in \kC$ and any $1$-dimensional subspace $x\subseteq C_1\cap C_2$, there is a circuit $C_3 \subseteq \kC$ such that $C_3 \subseteq C_1+C_2$ and $x\not\subseteq C_3$.
\end{itemize}
This is, at first sight, a more straightforward $q$-analogue of the axiom for classical matroids. For classical matroids, the two statements are equivalent, but we will see that for $q$-matroids they are not. We will see a similar issue with the axiom (O3) for open spaces in Remark \ref{Equivalenta}.

However, $\overline{\mbox{(C3)}}$ is a weaker version of the axiom (C3) we have proven above, as we will show. Let $C_1,C_2$ be distinct circuits and let $x$ be a $1$-dimensional space contained in $C_1\cap C_2$. Then there is a space $X\in \kL(E)$ of codimension $1$ that intersects trivially with $x$. Apply (C3) to $C_1,C_2$ and $X$: this gives a circuits $C_3\subseteq(C_1+C_2)\cap X$. This is clearly a circuit contained in $C_1+C_2$ that does not contain $x$. So (C3) implies $\overline{\mbox{(C3)}}$. The implication does not go the other way: it can be that $C_1,C_2\not\subseteq X$ but $C_1\cap C_2\subseteq X$. In that case, the statement above does not imply the existence of a circuit $C_3\subseteq(C_1+C_2)\cap X$. We illustrate this in the next example.
\end{Remark}

\begin{Example}[Example 10 of \cite{JP18}]
Let $E=\mathbb{F}_2^4$ and let $I\in \kL(E)$ be the subspace given by
\[ I=\left\langle\begin{array}{cccc} 1 & 0 & 0 & 1 \\ 0 & 1 & 1 & 0 \end{array}\right\rangle. \]
Let $\mathcal{I}$ be the family of subspaces of $E$ that contains $I$ and all subspaces of $I$. As is pointed out in \cite{JP18}, $\mathcal{I}$ satisfies the independence axioms (I1)-(I3) but not (I4). Let $\mathcal{C}_{\mathcal{I}}=\min(\opp(\mathcal{I}))$, that is, the family of `circuits' implied by $\mathcal{I}$. Let us examine $\mathcal{C}_{\mathcal{I}}$. It contains all $1$-dimensional subspaces of $E$ that are not in $I$; we call them \emph{loops}. Any $4$- and $3$-dimensional subspace of $E$ contains a loop, hence none of these is a member of $\mathcal{C}_{\mathcal{I}}$. Every $2$-dimensional subspaces of $E$ either contains a loop, or is equal to $I$, so none of these is a member of $\mathcal{C}_{\mathcal{I}}$. Hence $\mathcal{C}_{\mathcal{I}}$ only contains loops.

It is clear that $\mathcal{C}_{\mathcal{I}}$ satisfies the circuits axioms (C1) and (C2). Since all pairs of members of $\mathcal{C}_{\mathcal{I}}$ intersect trivially, $\mathcal{C}_{\mathcal{I}}$ satisfies $\overline{\mbox{(C3)}}$ as well. This shows that (C1), (C2) ,$\overline{\mbox{(C3)}}$ can not be a full axiom system for a $q$-matroid, as was also noted in the discussion after \cite[Theorem 64]{JP18}.

The family of subspaces $\mathcal{C}_{\mathcal{I}}$ does not satisfy the axiom (C3): for a counter example, take $C_1=\langle1100\rangle$, $C_2=\langle 0011\rangle$ and $X=\langle1001\rangle^\perp$. Then $(C_1+C_2)\cap X=\langle1111\rangle$ and this does not contain a member of $\mathcal{C}_{\mathcal{I}}$. This shows that (C3) is a stronger axiom than $\overline{\mbox{(C3)}}$.
\end{Example}

\section{Open Spaces and Flats}\label{OpensSect}
In this section, we discuss the axiomatic definition of 
open spaces and prove the cryptomorphism between open spaces and flats. We follow the same approach as in the previous section and call $\kO^*$ the family of \textbf{co-open spaces} of a $q$-matroid.

\begin{Theorem}
Let $\kO^*$ and $\kF$ be two families of subspaces of $E$ such that $\kO^*=\kF^\perp$. Then $(E,\mathcal{F})$ is a collection of flats if and only if $(E,\kO^*)$ is a collection of open spaces.
\end{Theorem}
\begin{proof}
Suppose $(E,\kF)$ is a collection of flats, so that it satisfies the flat axioms. Since $\kO^*=\kF^\perp$, we get that $\kO^*$ satisfies the open space axioms by taking orthogonal complements in all the flat axioms. Since $(\kF^\perp)^\perp=\kF$, we get the other implication by taking orthogonal complements again.
\end{proof}

The fact that a collection of co-open spaces determines $q$-matroid is the content of the following corollary.

\begin{Corollary}\label{OpenAxioms}
Let $(E,\kO^*)$ be a collection of open spaces and let $(E,\kF)$ be a collection of flats.
Suppose that $\kO^* = \kF^\perp$. Then both $(E,\kO^*)$ and $(E,\kF)$ each determine the same $q$-matroid $(E,r)$
such that $\kO^*$ is the collection of co-open spaces of $(E,r)$ and $\kF$ is the collection of flats of $(E,r)$.
\end{Corollary}
\begin{proof}
By Theorem \ref{th:bciflatsrk}$, (E,\kF)$ is a $q$-matroid whose family of flats is equal to $\kF$. The result now follows since $\kO^* = \kF^\perp$.
\end{proof}

As with cocircuits and circuits, co-open spaces are open spaces of the dual $q$-matroid.

\begin{Proposition}
The flats of a $q$-matroid $M=(E,r)$ are the orthogonal spaces of the open spaces of the dual $q$-matroid $M^*$. 
\end{Proposition}
\begin{proof}
In \cite{WINEpaper1}, it was proved that the lattice of flats is semimodular with the meet of two flats $F_1,F_2$ defined to be $F_1 \wedge F_2 := F_1 \cap F_2$ and the join $F_1 \vee F_2$ is defined to be the minimal flat containing $F_1 + F_2$, which is $\cl_r(F_1+F_2)$. The maximal flats of $M$ are the hyperplanes. \\
Dualizing to co-open sets, we have an {\em anti-isomorphism} and we have a semimodular lattice of open spaces, where, if $O_1,O_2 \in \kO$, $O_1 \wedge O_2= O_1+O_2$, while their meet is the maximal subspace contained in their intersection. Since the orthogonal complements of hyperplanes are cocircuits, it follows that every co-open space is the sum of cocircuits. By Proposition \ref{CircuitsAreCocirc}, cocircuits are circuits in $M^*$, hence sums of cocircuits are sums of circuits in $M^*$ and these are by definition open spaces.
\end{proof}

From this proposition it follows directly that the open spaces of a $q$-matroid are a collection of open spaces.

\begin{Corollary}
Let $(E,\kO^*)$ be a collection of co-open spaces of a $q$-matroid $M$. Then $(E,\kO)$ is the collection of open spaces of $M^*$.
\end{Corollary}

\begin{Remark}\label{Equivalenta}
    Consider the following open {\em set} axioms for classical matroids, for a collection $\kO$ of
    subsets of some ground set $S$ of finite cardinality $n$.
   
	\begin{itemize}
	    \item[(O1)] The empty set is a member of $\kO$.
	    \item[(O2)] If $O_1,O_2 \in \kO$ then $O_1 \cup O_2 \in \kO$.
		\item[(O3)] For each $O \in \kO$ and each subset $X \subset S$ of cardinality $n-1$ such that $O\nsubseteq X$, there exists a unique set $O' \in \kO$, such that $O' \subseteq X \cap O$ and $O'$ is covered by $O$ in $\kO$.
	\end{itemize}

	\begin{itemize}
		\item[(\ol{O3})] For each $O \in \kO$, if $O_1,\ldots,O_k \in \kO$ are all the sets in $\kO$ covered by $O$ in $\kO$, then $\bigcap_{i=1}^k O_i= \emptyset$.
	\end{itemize}
    The direct $q$-analogue of the axioms (O1)-(O3) given above are given by the open spaces axioms of Definition \ref{OpenSpaces}, while the axioms (O1), (O2) and (\ol{O3}) are the usual classical open space axioms.
	In fact, as we now show, the open set axioms (O1), (O2), (O3), are equivalent to (O1), (O2), (\ol{O3}).
	Let $M$ be a matroid with ground set $S$ of size $n$ and let $\kO$ be a collection of subsets of $S$.

     \paragraph{(O3) $\Rightarrow$ (\ol{O3}):}
		Assume that (O3) holds. Let $O \in \kO$ and let $O_1,\ldots,O_k$ be all the open sets covered by $O$ in $\kO$.
		Suppose that $\bigcap_{i=1}^k O_i$ is non-empty and so contains some element $h$. Let $X'=S \setminus \{h\}$. By (O3), there exists a unique open set $O' \subseteq X'\cap O = O \setminus \{h\}$ that is
		covered by $O$ in $\kO$. By construction, this set $O'$ does not contain $h$, which contradicts the assumption that $h$ is contained in the intersection of all such sets.
    
        \paragraph{(\ol{O3}) $\Rightarrow$ (O3):}
		Now assume that (\ol{O3}) holds. Let $O \in \kO$ and let $X$ be a subset of $S$ of cardinality $n-1$ such that $O \nsubseteq X$. Then $S = X \cup \{h\}$ for some $h \in S$. 
		Now suppose, towards a contradiction, that there is no subset of $X \cap O$ that is covered by $O$ in $\kO$. Then in particular, there no such set contained in $X$, so all sets covered by $O$ in $\kO$ contain $h$. However, this contradicts (\ol{O3}), which we have assumed by hypothesis. We deduce that (O3) holds.
	\end{Remark}
	
	A direct $q$-analogue of (\ol{O3}) is given by the following for a collection $\kO$ of subspaces of $E$.
	\begin{itemize}
	    \item[(\ol{O3})]
	    For each $O \in \kO$, if $O_1,\ldots,O_k \in \kO$ are all the subspaces in $\kO$ covered by $O$ in $\kO$, then $\bigcap_{i=1}^k O_i= \{0\}$.
	\end{itemize}
	However, even though (O3) and (\ol{O3}) are equivalent in the classical case, this cannot be said of their $q$-analogues, as the following example shows.

	\begin{Example}\label{ExOpens}
	We give an easy counterexample, coming from the $q$-matroid $M_6^*$, namely the dual of $M_6$ from Example \ref{M2s}. By dualizing the flats in Table \ref{Tabellona}, we see that the open spaces of the $q$-matroid $M_6^*$ are $0, \FF_2^6$ and the orthogonal complements of $G_1,\ldots,G_9$, namely $G_1^\perp,\ldots,G_9^\perp$.

	It can be easily observed that the set $L_{O'}=\{\{0\},G_1^\perp,\ldots,G_8^\perp,\FF_2^6\}$, which is the set of open spaces of $M_6$ excluding $G_9^\perp$, satisfies (O1), (O2), and (\ol{O3}), as we now show.
	Clearly, $\{0\} \in L_{O'}$. Since the $G_1^\perp,\ldots,G_8^\perp$ all have trivial pairwise intersections, their pairwise vector-space sums are all equal to $\FF_2^6$ and clearly the sum of any member $L_{O'}$ with 
	$\{0\}$ or $\FF_2^6$ is contained in $L_{O'}$ so that (O2) holds.
	Also (\ol{O3}) holds; the only nontrivial case to consider is that involving the open spaces covered by $\FF_2^6$, which are $G_1^\perp,\ldots,G_8^\perp$ and have trivial intersection.
	We will now show that $L_{O'}$ does not satisfy (O3). 
	Let $O=\FF_2^6$ and let $X:=G_9^\perp + \langle (1, 0, 0, 1, 0, 0),(1, 0, 0, 0, 0,1) \rangle $. 
	Then $X$ has codimension 1 in  $\FF_2^6$ and clearly $X\cap O=X$. 
	The only space in $L_{O'}$ in $X$ that is not covered by $O$ is the zero space and in particular, 
	it is not true that there is a unique open space covered by $O$ in $X\cap O=X$. Therefore (O3) fails for the collection $L_{O'}$.
\end{Example}

\section{Spanning and Non-spanning Spaces}\label{ToSpanOrNotToSpan}
In this short section, we discuss spanning and non-spanning spaces. We follow the same approach as the previous two sections. Therefore we prove the duality between independent and spanning spaces between and dependent and non-spanning spaces.

\begin{Proposition}\label{OppIndSp}
The orthogonal complements of the independent spaces of $M$ are the spanning spaces of $M^*$.
\end{Proposition}
\begin{proof}
By definition, an independent space $I$ has $r(I)=\dim(I)$. Applying the dual rank function to $I^\perp$ and $E$ gives that
\begin{align*}
r^*(I^\perp) &=\dim(I^\perp)-r(E)+r(I) \\
 &=\dim(E)-\dim(I)-r(E)+\dim(I) \\
 &=\dim(E)-r(E) \\
 &=r^*(E)
\end{align*}
and this is exactly saying that $I^\perp$ is a spanning space of $M^*$.
\end{proof}

In a similar fashion as the previous two sections, we can now prove that $\kS^*=\I^\perp$ is a collection of spanning spaces, and in combination with the proposition above we arrive at the following.

\begin{Corollary}\label{Spanningq-Matr}
Let $(E,\kS)$ be a collection of spanning spaces and let $(E,\I)$ be a collection of independent spaces.
Suppose that $\kS^\perp = \I$. Then both $(E,\kS)$ and $(E,\I)$ each determine the same $q$-matroid $(E,r)$
such that $\kS$ is the collection of spanning spaces of $(E,r)$ and $\I$ is the collection of independent spaces of $(E,r)$.
\end{Corollary}

We can repeat the very same reasoning for non-spanning spaces. In particular, spanning sets should be substituted by non-spanning spaces and independent spaces should be replaced by dependent spaces.
We get then the following.

\begin{Proposition}
The orthogonal complements of the dependent spaces of $M$ are the non-spanning spaces of $M^*$.
\end{Proposition}
\begin{proof}
Let $\kN^*$ be the non-spanning spaces of $M^*$. Then $\opp(\kN^*)=\kS^*$ are the spanning spaces of $M^*$. By Proposition \ref{OppIndSp}, these are the orthogonal complements of the independent spaces of $M$. The result now follows because $\I=\opp(\kD)$. See also Figure \ref{diagram2}.
\end{proof}

\begin{Corollary}\label{NOSpanningq-Matr}
Let $(E,\mathcal{N})$ be a collection of non-spanning spaces and let $(E,\mathcal{D})$ be a collection of dependent spaces.
Suppose that $\mathcal{N}^\perp = \mathcal{D}$. Then both $(E,\mathcal{N})$ and $(E,\mathcal{D})$ each determine the same $q$-matroid $(E,r)$
such that $\mathcal{N}$ is the collection of non-spanning spaces of $(E,r)$ and $\mathcal{D}$ is the collection of dependent spaces of $(E,r)$.
\end{Corollary}


\section{Acknowledgements}

This paper stems from a collaboration that was initiated at the Women in Numbers Europe (WIN-E3) conference,
held in Rennes, August 26-30, 2019. The authors are very grateful to the organisers: Sorina Ionica, Holly Krieger, and Elisa Lorenzo Garc\'ia, for facilitating their participation at this workshop, which was supported by the Henri Lebesgue Center, the Association for Women in Mathematics (AWM) and the Clay Mathematics Institute (CMI).

\bibliographystyle{alpha}
\bibliography{References} 
 
\end{document}